\documentclass[reqno]{amsart}

\usepackage{amsmath,amssymb,amsthm,comment,enumerate,mathrsfs}
\usepackage{hyperref}

\swapnumbers

\numberwithin{equation}{section}

\newtheorem{Thm}{Theorem}[section]
\newtheorem{Prop}[Thm]{Proposition}
\newtheorem{Lem}[Thm]{Lemma}
\newtheorem{Cor}[Thm]{Corollary}
\theoremstyle{definition}
\newtheorem{Rem}[Thm]{Remark}
\newtheorem{Def}[Thm]{Definition}
\newtheorem{Expl}[Thm]{Example}

\newcommand{\R}{\mathbb{R}}

\newcommand{\diam}{\operatorname{diam}}

\newcommand{\eps}{\varepsilon}

\renewcommand{\rho}{\varrho}

\title{Injective Convex Polyhedra}

\author{Ma\"el Pav\' on}
\address{Department of Mathematics, ETH Z\"urich, 8092 Z\"urich, Switzerland}
\email{mael.pavon@math.ethz.ch}

\date{\today}

\begin{document}


\maketitle

\begin{abstract}
It was shown by Nachbin in 1950 that an $n$-dimensional normed space $X$ is injective or equivalently is an absolute 1-Lipschitz retract if and only if $X$ is linearly isometric to $l_\infty^n$ (i.e., $\R^n$ endowed with the $l_{\infty}$-metric). We give an effective convex geometric characterization of injective convex polyhedra in~$l_{\infty}^n$. As an application, we prove that if the set of solutions to a linear system of inequalities with at most two variables per inequality is non-empty, then it is injective when endowed with the $l_{\infty}$-metric.
\end{abstract}

\section{Introduction}

We call a metric space $X$ \textit{injective} if for any metric spaces $A,B$ such that there exists an isometric embedding $i \colon A \to B$ and for any $1$-Lipschitz (i.e., distance nonincreasing) map $f \colon A \to X$, there is a $1$-Lipschitz map $g \colon B \to X$ satisfying $g \circ i = f$ (cf. \cite[Section~9]{AdaHS} for the general categorical definition). In particular, it follows from a result of Nachbin that a real normed space $X$ is injective in the the category of metric spaces if and only if $X$ is injective in the category of linear normed spaces.

The purpose of the present work is to provide an effective characterization of injective convex polyhedra in $l_{\infty}^n$ by proving an easy combinatorial criterion. It is important to note that only the case of the  $l_{\infty}$-metric is relevant since if a convex polyhedron $P \subset \R^n$ with non-empty interior is injective for some norm $\left\| \cdot \right\|$ on $\R^n$, then considering an increasing sequence of rescalings of $P$ whose union is equal to $\R^n$, it follows by Lemma \ref{Lem:l12} that the  space $(\R^n$, $\left\| \cdot \right\|)$ is itself injective and by \cite[Theorem~3]{Nac}, which states that an $n$-dimensional normed space $X$ is injective if and only if $X$ is linearly isometric to $l_\infty^n$, it follows that $(\R^n$, $\left\| \cdot \right\|)$ is isometric to $l_{\infty}^n$.

Note at this point that linear subspaces of injective normed spaces need not be injective. A straightforward example is the plane 
\begin{equation}\label{eq:e0000000}
V := \{x \in l_\infty^3 : x_1 + x_2 + x_3 = 0\}
\end{equation}
which is not injective since it can be easily seen that the unit ball of $V$ is an hexagon and thus $V$ cannot be isometric to $l_\infty^2$. Furthermore, Example \ref{Expl:expl1} exhibits a non-injective convex polyhedron with injective supporting hyperplanes and Example \ref{Expl:expl2} an injective convex polyhedron with a non-injective face.

It was noted in \cite{Lan} that a good characterization of injective polytopes is missing. The present work gives a solution to this problem. We shall start by giving in the next section a characterization of injective affine subspaces of $l_\infty^n$ and as a consequence we shall obtain an easy injectivity criterion for hyperplanes, namely if $\nu \in \R^n \setminus \{ 0\}$, then the hyperplane 
\begin{equation}\label{eq:e00000001}
X := \{ x \in \R^n : x \cdot \nu = 0 \} \subset l_{\infty}^n
\end{equation}
(where $x \cdot y$ denotes the standard scalar product on $\R^n$) is injective if and only if 
\begin{equation}\label{eq:e00000}
\left\| \nu \right\|_{1} \le 2 \left\| \nu \right\|_{\infty}.
\end{equation}
For $\alpha \in \R$ and $\emptyset \neq A,B \subset \R^n$, we define $\alpha A, A+B, A-B \subset \R^n$ in the obvious way and we set $[a,b]A:=\bigcup_{\alpha \in [a,b]} \alpha A$.
For a convex polyhedron $\emptyset \neq P \subset \R^n$ and a point $p \in P$, the \textit{tangent cone} $\mathrm{T}_pP$ is given by  
\[
\mathrm{T}_pP := \bigcup_{m \in \mathbb{N}} P_{p,m} \ \text{ where } \ P_{p,m} := p + m (P - p).
\]
The effective characterization we are aiming at will be obtained in two steps. First, we shall prove that injectivity follows from a local injectivity property namely injectivity of tangent cones. It is no restriction to assume that the interior of $P$ satisfies $\mathrm{int}(P) \neq \emptyset$ in the next theorem:
\begin{Thm}\label{Thm:t00001}
Let $P\subset l_{\infty}^n$ be a convex polyhedron such that $\mathrm{int}(P) \neq \emptyset$. Then, the following are equivalent:
\begin{enumerate}[$(i)$]
\item $P$ is injective.
\item $\mathrm{T}_pP$ is injective for every $p \in \partial P$.
\end{enumerate}
\end{Thm}

By a \textit{convex polyhedron} in $\R^n$ we mean a finite intersection of closed half-spaces. Closed half-spaces are just called half-spaces when no ambiguity arises. A \textit{convex polytope} is then a compact convex polyhedron. A \textit{cone} $C$ is a subset of $\R^n$ such that $x \in C$ implies $\lambda x \in C$ for any $\lambda \ge 0$. Convex polyhedra which are additionally cones are called \textit{convex polyhedral cones}. If $C$ is a convex polyhedral cone and $x \in \R^n$, the \textit{apex} $\mathrm{apex}(x + C)$ of a translate of $C$ is defined as the affine space $x + V$ where $V$ is the biggest linear subspace of $\R^n$ contained in $C$. It is easy to see that $\mathrm{T}_pP - p$ is a convex polyhedral cone. In the sequel, the relative interior of a subset $S$ is denoted by $\mathrm{relint}(S)$. The \textit{dimension} of a convex polyhedron $P \subset \R^n$ is the dimension of its affine hull. One has $\mathrm{int}(P) \neq \emptyset$ if and only if $\mathrm{dim}(P) = n$ and in this case, $F$ is a \textit{facet} of $P$ if and only if $F$ is a face of $P$ and $\mathrm{dim}(F) = n-1$. Let us denote by $\mathrm{Faces}(P)$ and $\mathrm{Facets}(P)$ the set of non-empty faces and the set of facets of $P$ respectively, for any subset $S \subset \R^n$ let $\mathrm{Faces}(P,S):= \{ F \in \mathrm{Faces}(P) : F \cap S \neq~\emptyset \}$ and let $\mathrm{Faces}(P,S)^{c}$ be the complement of $\mathrm{Faces}(P,S)$ in $\mathrm{Faces}(P)$. Moreover, $\mathrm{Facets}^*(P,S):= \{ F \in \mathrm{Facets}(P) : \mathrm{relint}(F) \cap S \neq \emptyset \}$. Note that the closed unit ball $B(0,1) \subset l_{\infty}^n$ is nothing but the $n$-hypercube $[-1,1]^{n}$ endowed with the $l_{\infty}$-metric. The following theorem characterizes injective convex polyhedral cones:

\begin{Thm}\label{Thm:t2}
A convex polyhedral cone $C \subsetneq l_{\infty}^n$ with $\mathrm{int}(C) \neq \emptyset$ is injective if and only if the following hold:
\begin{enumerate}[$(i)$]
\item $\mathrm{T}_pC$ is injective for every $p \in \partial C \setminus \mathrm{apex}(C)$.
\item There is a facet $F \in \mathrm{Facets}^*([-1,1]^{n},C)$ such that $-F \notin \mathrm{Facets}^*([-1,1]^{n},C)$.
\end{enumerate}
\end{Thm}

It follows from Theorem \ref{Thm:t2} in the case where $\partial C \setminus \mathrm{apex}(C)=\emptyset$ or equivalently when $C$ is a half-space, that \eqref{eq:e00000} is an injectivity criterion for the half-spaces having the hyperplane $X$ as in \eqref{eq:e00000001} as boundary. For $p \in \partial C \setminus \mathrm{apex}(C) \neq \emptyset$, the dimension of $\mathrm{apex}(\mathrm{T}_pC)$ is strictly bigger than that of $\mathrm{apex}(C)$ and making repeated use of Theorem~\ref{Thm:t2} on tangent cones, one thus easily obtains:

\begin{Cor}\label{Cor:corollary1}
A convex polyhedron $P \subsetneq l_{\infty}^n$ with $\mathrm{int}(P) \neq \emptyset$ is injective if and only if for every $p \in \partial P$, the convex polyhedral cone $K := \mathrm{T}_pP - p$ satisfies $(ii)$ in Theorem \ref{Thm:t2}, which means that there is a facet $F \in \mathrm{Facets}^*([-1,1]^{n},K)$ such that $-F \notin \mathrm{Facets}^*([-1,1]^{n},K)$. 
\end{Cor} 

There are several equivalent characterizations of injective metric spaces and one of them is \textit{hyperconvexity} (cf. \cite{AroP}). We call a metric space $X$ \textit{hyperconvex} if for every family $\{(x_i,r_i) \}_{i \in I}$ in $X \times \R$ satisfying $r_i + r_j \ge d(x_i,x_j)$ for all $(i,j) \in I \times I$, one has $\bigcap_{i \in I} B(x_i,r_i) \neq \emptyset$ (with the convention that the intersection equals $X$ itself if $I = \emptyset$) where $B(x,r)$ will denote throughout the text, a closed ball in the contextually relevant metric (whereas open balls will be denoted by $U(x,r)$). Furthermore, if $Y \subset Z$ with $Z$ being injective and if there is a $1$-Lipschitz retraction $r \colon Z \to Y$ (i.e., $r \in \mathrm{Lip}_1(Z,Y)$ and $r|_Y = \mathrm{id}_Y$), then $Y$ is injective (this follows immediately from the definition of injectivity given above). The following two examples show that the characterization we are looking for requires more effort than one would think at first sight:

\begin{Expl}\label{Expl:expl1}
Consider the half-spaces 
\[
H := \{ x \in l_{\infty}^4 : x_1 \ge 0\}
\]
and 
\[
H' := \{ x \in l_{\infty}^4 : x_1 \le \frac{1}{3} (x_2 + x_3 + x_4) \}.
\]
Note that it is easy to see that both $H$ and $H'$ are injective by considering in each case the $1$-Lipschitz retraction given by mapping each point in the complement to the unique corresponding point on the boundary so that all coordinates but the first remain unchanged and then extending by the identity. Moreover, both $\partial H$ and $\partial H'$ are injective by \eqref{eq:e00000}. However, it is easy to see that $P := H \cap H' \subset l_{\infty}^4$ is not injective by considering the three points 
\[
\{p,p',p''\} := \{ (0,0,0,0), (0,0,-2,2),(0,-2,0,2)\} \subset \partial H \cap \partial H' \subset P,
\]
note that
\[
I := B(p,1) \cap B(p',1) \cap B(p'',1) = \{ (t,-1,-1,1) : t \in [-1,1] \},
\]
hence $I \cap P = \emptyset$. Thus $P$ is not hyperconvex and therefore not injective.
\end{Expl}
Next, we have:
\begin{Expl}\label{Expl:expl2}
Consider the injective half-space $H'$ defined above, let further $H'' := \{ x \in l_{\infty}^4 : x_1 \le 0\}$ and 
\[
P' := H' \cap H'' \subset l_{\infty}^4.
\]
Note that the face 
\[
F := \partial H' \cap \partial H'' \subset l_{\infty}^4
\]
of $P'$ is not injective since 
\[
F = \{ x \in l_{\infty}^4 : x_1 = 0, x_2 + x_3 + x_4 = 0 \}
\]
is isometric to \eqref{eq:e0000000} which is not injective as we already noted. Let us now however show that $P'$ is injective by defining an explicit $1$-Lipschitz retraction $r$ of $l_{\infty}^4$ onto $P'$. Let $\rho \in \mathrm{Lip}_1(l_\infty^{4}, \R)$ be the map 
\[
(x_1,\dots,x_4) \mapsto \frac{1}{3} (x_2 + x_3 + x_4).
\] 
Now, let $r \colon l_{\infty}^4 \to P'$ be given by
\[
(x_1,\dots,x_4) \mapsto (\min \{ x_1, 0, \rho(x) \},x_2,x_3,x_4 )
\]
and note that $r$ is the desired $1$-Lipschitz retraction.
\end{Expl}

Finally, we use Corollary \ref{Cor:corollary1} and a theorem of Shostak cf. \cite{Sho}, to prove:

\begin{Cor}\label{Cor:corollary2}
Consider two maps $f,g \colon \{1,\dots,m\} \to \{1,\dots,n\}$ and $a_i,b_i,c_i \in \R$ for $i \in \{1,\dots,m\}$ such that
\[
P := \bigcap_{i \in \{1,\dots,m\}} \left \{ x \in \R^{n} : a_i x_{f(i)} + b_i x_{g(i)} \ge c_i \right \} \neq \emptyset.
\]
Then, $P \subset l^n_{\infty}$ satisfies the assumptions of Corollary \ref{Cor:corollary1} and is therefore injective.
\end{Cor} 

\section{Injective Linear Subspaces in $l_{\infty}^n$}\label{sec:s2}
Consider for $i \in I_n := \{1,\dots,n\}$ the linear isometry
\[
\mu_i \colon l_{\infty}^n \to l_{\infty}^n, \ (x_1,\dots,x_n) \mapsto (x_1,\dots,x_{i-1},-x_i,x_{i+1},\dots, x_n)
\]
and the $1$-Lipschitz linear map
\[
\pi_i \colon l_{\infty}^n \rightarrow \R, \ (x_1,\dots,x_n) \mapsto x_i.
\]
Moreover, let us denote by $\{e_1,\dots,e_n\}$ the standard basis of $\R^n$. Injective convex polyhedra were also studied in \cite{Moe}. Note that Theorem \ref{Thm:001} and \ref{Thm:t000000000} as well as Lemma~\ref{Lem:l12} in the next section already appear in \cite{Moe}. Our proof of Theorem \ref{Thm:001} is however more elementary.
\begin{Thm}\label{Thm:001}
Let $\emptyset \neq X \subset l_{\infty}^n$ be a linear subspace and let $k:= \mathrm{dim}(X)$. Then, the following are equivalent:
\begin{enumerate}[$(i)$]
\item $X$ is injective.
\item There is a subset $J \subset I_n$ with $|J|=k$ such that for any $i \in I_n \setminus J$ there exist real numbers $\{c(i,j)\}_{j \in J}$ such that $\sum_{j \in J} |c(i,j)| \le 1$ and such that
\[
X = \Biggl \{ x \in l_{\infty}^n : \forall i \in I_n \setminus J \ , \ x_i = \sum_{j \in J} c(i,j)x_j \Biggr \}.
\]
\end{enumerate}
\end{Thm}
\begin{proof}
Assume first that $(ii)$ holds. Assume for simplicity that $J=\{1,\dots,k\}$. Let us define the map $L \colon l_\infty^k \to l_\infty^n$ such that for any $(y_1,\dots,y_k) \in l_\infty^k$,
\[
L(y) := \Biggl( y_1, \dots, y_k, \sum_{j = 1}^k c(k+1,j)y_j,\dots, \sum_{j = 1}^k c(n,j)y_j \Biggr).
\]
It is then easy to see that $L$ is an isometric embedding with $L(l_\infty^k)=X$. It follows that $X$ and $ l_\infty^k$ are isometric and thus $X$ is injective.

Assume now that $(i)$ holds, there consequently exists a linear isometric embedding $L \colon l_\infty^k \to X \subset l_\infty^n$  (see the Introduction). In particular,
\begin{equation}\label{eq:e001}
\left\|  L(e_j)  \right\|_{\infty} = 1
\end{equation}
and 
\begin{equation}\label{eq:e002}
\left\|  L(\sigma e_j + \tau e_l )  \right\|_{\infty} = 1
\end{equation}
for $(j,l) \in I_k \times I_k$ with $j \neq l$ (where $I_k:=\{1,\dots,k\}$) and $\sigma,\tau \in \{\pm 1\}$. Now, \eqref{eq:e001} implies for $j \in I_k$ the existence of some $f(j) \in I_n$ such that $|(\pi_{f(j)} \circ L)(e_j)| =1$; replacing $L$ by $L \circ \mu_j$ if necessary, we can assume without loss of generality that 
\begin{equation}\label{eq:e003}
(\pi_{f(j)} \circ L)(e_j) = 1
\end{equation}
for any $j \in I_k$. Therefore, \eqref{eq:e003} together with \eqref{eq:e002} imply that $(\pi_{f(j)} \circ L)(e_l) = 0$ for $(j,l) \in I_k \times I_k$ with $j \neq l$ and thus $f$ is injective. We summarize by writing $(\pi_{f(j)} \circ L )(e_l) = \delta_{jl}$. Now, we can assume for simplicity that $f(j) = j$ for any $j \in I_k$ hence in particular $J := f(I_k) = \{1,\dots,k\}$ and
\begin{equation}\label{eq:e004}
(\pi_{j} \circ L )(e_l) = \delta_{jl}.
\end{equation}
It follows that there are $c(k+1,j), \dots, c(n,j) \in \R$ such that
\[
L(e_j) = \bigl(0,\dots,0,1,0,\dots,0,c(k+1,j),\dots,c(n,j) \bigr),
\]
where the first $k$ entries of $L(e_j)$ are zero except the $j$-th one. For any $(\sigma_1,\dots,\sigma_k) \in \{ \pm 1 \}^{k}$, one has by linearity
\[
\left\|  \sum_{j=1}^k \sigma_j L(e_j)   \right\|_{\infty} = \left\|  \sum_{j=1}^k \sigma_j e_j \right\|_{\infty} = 1.
\]
Inserting successively appropriate values for $(\sigma_1,\dots,\sigma_k)$ in the above equality, one obtains for any $i \in I_n \setminus J = \{k+1,\dots,n\}$,
\[
\sum_{j=1}^k \left| c(i,j) \right| \le 1.
\]
Since $X=L(l_\infty^k)$, there are for any $x \in X$ real numbers $c_1,\dots,c_k \in \R$ such that $x = \sum_{l=1}^k c_l L(e_l)$. For any $j \in I_k$, it follows from \eqref{eq:e004} that
\[
x_j = \pi_j(x) = \sum_{l=1}^k c_l (\pi_j \circ L)(e_l) = \sum_{l=1}^k c_l \delta_{jl} = c_j.
\]
Hence finally
\[
x = \sum_{j=1}^k x_j L(e_j) = \Bigl(x_1,\dots,x_k,\sum_{j=1}^k c(k+1,j) x_j ,\dots,\sum_{j=1}^k c(n,j) x_j \Bigr).
\]
This proves that $(ii)$ holds and concludes the proof.
\end{proof}

The next theorem is an immediate consequence of Theorem \ref{Thm:001}:

\begin{Thm}\label{Thm:t000000000}
Let $\nu \in \R^n \setminus \{ 0\}$. The hyperplane $X = \{ x \in \R^n : x \cdot \nu = 0 \} \subset l_{\infty}^n$ is injective if and only if $\left\| \nu \right\|_{1} \le 2 \left\| \nu \right\|_{\infty} $.
\end{Thm}

\begin{proof}
Assume first that $X$ is injective. By Theorem \ref{Thm:001}, there is some $i \in I$ such that 
\[
X = \Biggl \{ x \in l_{\infty}^n :  -x_i + \sum_{j \in I \setminus \{i\}} c(i,j)x_j = 0 \Biggr \}.
\]
with $\sum_{j \in I \setminus \{i\}} |c(i,j)| \le 1$. Define now $\nu$ so that $\nu_j:= c(i,j)$ if $j \neq i$ and $\nu_i:=-1$. Note that $\nu$ is a normal vector of $X$ and satisfies $\left\| \nu \right\|_{1} \le 2 \left\| \nu \right\|_{\infty} $.

For the other implication, let $\nu$ a normal vector of $X$ satisfying $\left\| \nu \right\|_{1} \le 2 \left\| \nu \right\|_{\infty} $ and assume without loss of generality that $\left\| \nu \right\|_{\infty} = 1$; hence, $\left\| \nu \right\|_1 \le 2$. There is  $i \in I$ such that $|\nu_i| = 1$ and assume additionally without loss of generality that $\nu_i = -1$. Thus $ \sum_{j \in I \setminus \{i\}} |\nu_j| \le 1$ and $x \cdot \nu = -x_i + \sum_{j \in I \setminus \{i\}} \nu_j x_j$, hence we can apply Theorem \ref{Thm:001} to
\[
X = \Biggl \{ x \in l_{\infty}^n :  -x_i + \sum_{j \in I \setminus \{i\}} \nu_j x_j = 0 \Biggr \},
\]
to obtain that $X$ is injective. This concludes the proof of the theorem.
\end{proof}

\section{Tangent cones of Injective Convex Polyhedra in $l_{\infty}^n$}
Throughout the text, we shall call a sequence of sets $(X_m)_{m \in \mathbb{N}}$ increasing if and only if $X_m \subset X_{m+1}$ for $m \in \mathbb{N}$ whereas it will be called decreasing if the reverse inclusions hold. 

\begin{Lem}\label{Lem:l12}
Let $\emptyset \neq S \subset  l_{\infty}^n$ be a closed subset. Then, the following are equivalent:
\begin{enumerate}[(i)]
\item  $S$ is injective.
\item There is $x \in S$ such that $S \cap B(x,r) $ is injective for any $r \in (0,\infty)$.
\item There is an increasing sequence $(X_m)_{m \in \mathbb{N}}$ of injective subsets of $S$ such that $S = \bigcup_m X_m$.
 \end{enumerate}
\end{Lem}
\begin{proof}
We shall only prove that $(iii)$ implies $(i)$ since the other implications follow immediately from the definitions. In order to do so, we shall prove that $(iii)$ implies that $S$ is hyperconvex. Consider a family $\{ (x_{\alpha},r_{\alpha} ) \}_{\alpha \in A}$ in $S \times \R$ such that $r_{\alpha} + r_{\beta} \ge \left\| x_{\alpha} - x_{\beta} \right\|_{\infty}$ for any $(\alpha,\beta ) \subset A \times A$. Pick $\gamma \in A$ arbitrarily and let $m_0 \in \mathbb{N}$ be such that $x_\gamma \in  X_{m_0}$. Consider a sequence $(A_m,y_m)_{m \in \mathbb{N}}$ such that
\[
A_m := \bigl \{ \alpha \in A : x_{\alpha} \in X_{m+m_0} \bigr \} 
\]
and
\[ 
y_m \in S \cap \bigcap_{\alpha \in A_m} B(x_{\alpha},r_{\alpha}),
\] 
noting that $X_{m+m_0} \cap \bigcap_{\alpha \in A_m} B(x_{\alpha},r_{\alpha})\neq \emptyset$ hence $S \cap \bigcap_{\alpha \in A_m} B(x_{\alpha},r_{\alpha})\neq \emptyset$. Since $S$ is closed and $(y_m) \subset S \cap B(x_{\gamma},r_{\gamma})$, it follows that there is a convergent subsequence $(y_{m_l})$ such that $y_{m_l} \to y \in S \cap B(x_{\gamma},r_{\gamma})$. Thus, $y \in S \cap \bigcap_{\alpha \in A} B(x_{\alpha},r_{\alpha})$. 
This proves that $S$ is hyperconvex and finishes the proof of the lemma.
\end{proof}

We shall make use in the proof of Theorem~\ref{Thm:t00001} of the following (cf.~\cite{Zie}):

\begin{Thm}\label{Thm:t001}
$S \subset \R^n$ is a convex polyhedron if and only if there is a convex polytope $Q$ and a convex polyhedral cone $C$ such that 
\[
S = Q + C.
\]
\end{Thm}

For an $n$-dimensional polyhedron $P$ and for $k \in \{0,1,\dots,n-1\}$, let $\mathrm{Faces}_k(P)$ denote the set of $k$-dimensional faces of $P$ and let $\partial^k P$ be the union of all elements of $\mathrm{Faces}_k(P)$. We shall use the notation $d(A,B) := \inf_{ (a,b) \in A \times B} \left\| a -b \right\|_{\infty}$ for two subsets $\emptyset \neq A,B \subset l_{\infty}^n$. The open $\delta$-neighborhood $\bigcup_{a \in A} U(a,\delta)$ of $A$ will be denoted by $N(A,\delta)$. 

\begin{proof}[Proof of Theorem~\ref{Thm:t00001}]
By Lemma \ref{Lem:l12} and by definition of $\mathrm{T}_pP$ it immediately follows that $(i)$ implies $(ii)$. Assume now that $(ii)$ holds. Let us consider an enumeration $\{F_j\}_{j \in \{1,\dots,N\}}$ of $\mathrm{Faces}(P) \setminus \{P\}$. For $j \in \{1,\dots,N\}$ we consider an arbitrary point $p^j \in \mathrm{relint}(F_j)$ and a corresponding $1$-Lipschitz retraction $\rho_j \colon l_{\infty}^n \to \mathrm{T}_{p^j}P$. For $p \in \partial P$, let
\[
\eps_p := \sup \{  \eps \in (0,\infty]: U(p,\eps) \cap \mathrm{T}_pP = U(p,\eps) \cap P \}.
\]
Note that if $\eps_p = \infty$ for some $p$, then $P = \mathrm{T}_pP$ and thus $P$ is injective. Otherwise, we proceed inductively to show that there is a $\delta > 0$ such that 
\[
P \cup N(\partial P,\delta) \subset P \cup \bigcup_{p \in \partial P} U(p,\eps_p).
\]
Suppose $F \in \mathrm{Faces}_k(P)\setminus \{P\}$, for $k = 0$ we set $c(F) := F$ and for $k \ge 1$:
\[
c(F)
:= F \setminus N(\partial^0 P \cup \partial^1 P \cup \cdots \cup \partial^{k-1} P, \delta^{(k)}/2).
\]
Moreover,
\[
\eps^{(k+1)} := \min_{F \in \mathrm{Faces}_k(P)} \Biggl[ \frac{1}{2} \min_{ \mathrm{Faces}(P) \ni F' \nsupseteq F } d(c(F),F') \Biggr].
\]
By Theorem \ref{Thm:t001} it is easy to see that $d(P',P'') > 0$ for any two disjoint convex polyhedra $\emptyset \neq P',P'' \subset l_{\infty}^n$ and thus $\eps^{(k+1)} > 0$. Furthermore, let $\delta^{(0)}:=\eps^{(1)}$ and
\[
\delta^{(k+1)} :=\min \left \{ \eps^{(k+1)}, \frac{\delta^{(k)}}{2}  \right \}.
\]
Moreover, we set $A^0 := \partial^0 P$ and for $k \ge 1$:
\[
A^k := \bigcup_{F \in \mathrm{Faces}_k(P)} c(F) = \partial^k P \setminus N(\partial^0 P \cup \partial^1 P \cup \cdots \cup \partial^{k-1} P,\delta^{(k)}/2).
\]
It follows by construction that for any $p \in A^k$ and any $F \in \mathrm{Faces}(P)$, one has $U(p,\delta^{(k+1)}) \cap F \neq \emptyset$ if and only if $p \in F$. Hence $U(p,\delta^{(k+1)}) \cap \mathrm{T}_pP = U(p,\delta^{(k+1)}) \cap P$ for any $p \in A^k$. It follows by induction that
\[
\bigcup_{p \in \partial P} U(p,\delta^{(n)}) = N(\partial^{n-1}P,\delta^{(n)}) \subset \bigcup_{k=0}^{n-1} N(A^k,\delta^{(k+1)}).
\]
This shows that $\delta := \delta^{(n)} > 0$ satisfies $P \cup N(\partial P,\delta) \subset P \cup \bigcup_{p \in \partial P} U(p,\eps_p)$. It is now easy to see that we obtain a $1$-Lipschitz retraction $\rho \colon N(P,\delta) \to P$ by setting $\rho := \widetilde{\rho}|_{N(P,\delta)}$ where $\widetilde{\rho} := \rho_1 \circ \cdots \circ \rho_N$. By Theorem \ref{Thm:t001}, there is a convex polytope $Q$ and a polyhedral cone $C$ such that $P = Q + C$. We can assume without loss of generality that $0 \in \mathrm{int}(Q)$. We can set $\kappa := 1 + \frac{\delta}{2 \diam(Q)}$ and since $\kappa P = \kappa Q + C$ it follows that $\kappa P \subset N(P,\delta)$. 
By iteration, we obtain a sequence $\{(\rho^m, P^m)\}_{m \in \mathbb{N} }$ of rescalings 
$P^m := \kappa^m P$ of $P$ and corresponding $1$-Lipschitz retractions $\rho^m \colon P^{m} \to P^{m-1}$ by setting $\rho^m(\kappa x) := \kappa \rho^{m-1}(x)$ for $m \ge 2$ and $\rho^{1}:=\rho|_{\kappa P}$. Finally, we can define the $1$-Lipschitz retraction $r \colon l_{\infty}^n \to P$ as an inverse limit map for the system $\{(\rho^m, P^m)\}_{m \in \mathbb{N} }$, that is $r(x) := (\rho^1 \circ \cdots \circ \rho^m)(x)$ where $m$ is the smallest natural such that $x \in P^{m}$. It follows that $P$ is injective.
\end{proof}

Let us consider a simple example to show that it is necessary in the above proof to argue locally before extending to increasing rescalings.

\begin{Expl}
Consider $Q:= [-2,0] \times [-2,0] = B((-1,-1),1) \subset l_{\infty}^2$. We enumerate the tangent cones of $Q$ as follows; for $k \in \{1,2,3,4\}$:
\[
C_{2k-1} := \mathrm{T}_{p_k}Q \text{ where } (p_1,\dots,p_4):=((-2,-2),(-2,0),(0,0),(0,-2)),
\]
\[
C_{2k} := \mathrm{T}_{q_k}Q \text{ where } (q_1,\dots,q_4)=((-2,-1),(-1,0),(0,-1),(-1,-2)).
\]
Consider corresponding $1$-Lipschitz retractions such that
\begin{align*}
&\rho_2(x_1,x_2):= (-x_1-4,x_2)\text{ if } x_1 < -2 \text{ and } \rho_6(x_1,x_2):= (-x_1,x_2) \text{ if } x_1 > 0, \\
&\rho_8(x_1,x_2):= (x_1,-x_2-4) \text{ if } x_2 < -2 \text{ and } \rho_4(x_1,x_2):= (x_1,-x_2) \text{ if } x_2 >0
\end{align*}
and extend $\rho_2$, $\rho_4$, $\rho_6$ and $\rho_8$ by the identity. Finally, we set for odd indices: $\rho_{1}:= \rho_{2} \circ \rho_{8}$ and $\rho_{2k-1}:= \rho_{2k} \circ \rho_{2k-2}$ for $k \neq 1$. It is then easy to see that $(\rho_8 \circ \cdots \circ \rho_1)((-10,-10)) = (-6,2) \notin Q$.
\end{Expl}

\begin{Rem}
Note that it is enough to assume that the minimal (for the inclusion) tangent cones of $P$ are injective. Hence, letting $P$ be a convex polyhedron with non-empty interior, the following are equivalent:
\begin{enumerate}[$(i)$]
\item $P$ is injective.
\item All minimal tangent cones of $P$ are injective.
\end{enumerate}
\end{Rem}


\section{Systems of Inequalities}
For $i \in I_n := \{1,\dots,n\}$ let $\widehat{\pi}_i \in \mathrm{Lip}_1(l_\infty^{n},l_\infty^{n-1})$ denote the map
\[
(x_1,\dots,x_n) \mapsto (x_1,\dots,\widehat{x}_i,\dots,x_n) 
\]
and recall that $\pi_i \in \mathrm{Lip}_1(l_\infty^{n},\R)$ denotes the map $(x_1,\dots,x_n) \mapsto x_i$. We start with a proposition very similar to an idea originally from \cite{Des}.

\begin{Prop}\label{Prop:p1}
Let $I \subset I_n$, $\mathfrak{R} := \{\underline{r}_{i} : i \in I\} \cup \{ \overline{r}_{i} : i \in I \} \subset \mathrm{Lip}_1(l_\infty^{n-1}, \R)$ and 
\[
Q := \Bigl \{ x \in l_\infty^{n} : \forall i \in I , \ (\underline{r}_i \circ \widehat{\pi}_i)(x)  \le x_i \le (\overline{r}_i \circ \widehat{\pi}_i)(x) \Bigr \}.
\]
Assume that:
\begin{enumerate}[$(i)$]
\item $Q \neq \emptyset$;
\item for any $i \in I$, $\underline{r}_i \le \overline{r}_{i}$.
\end{enumerate}
It follows that $Q$ is injective.
\end{Prop}

\begin{proof}
We first show the statement in the case $\mathfrak{R} \subset \mathrm{Lip}_{\lambda}(l_\infty^{n-1}, \R)$ for some $\lambda \in [0,1)$.
For $i \in I$, let us define $\rho_i \in \mathrm{Lip}_{1}(l_\infty^{n}, l_\infty^{n})$ by setting
\[
\rho_i(x) := \Bigl(x_1,\dots,x_{i-1}, \min \bigl \{  (\overline{r}_i \circ \widehat{\pi}_i)(x) , \max \{ x_i, (\underline{r}_i \circ \widehat{\pi}_i)(x)  \}  \bigr \} , x_{i+1}, \dots, x_n \Bigr)
\]
for any $x \in l_\infty^n$. Consider an enumeration $I = \{i_1,\dots,i_N\}$. Moreover, set 
\[
G_j := \rho_{i_j} \circ \cdots \circ \rho_{i_1},
\]
$G_0 := \mathrm{id}_{l_\infty^n}$ and 
\[
T := G_N = \rho_{i_N} \circ \cdots \circ \rho_{i_1}.
\]
Fix now $x \in l_\infty^n$. We show that $(T^m(x))_{m \in \mathbb{N}}$ converges to a fixed point of $T$. Let us define the maps $\{f_{i_j}\}_{i_j \in I} \subset \mathrm{Lip}_{\lambda}(l_\infty^{n},\R)$ by 
\[
f_{i_j} \ : \ y \mapsto \min \Bigl \{ (\overline{r}_{i_j}\circ \widehat{\pi}_{i_j})(y), \max \bigl \{\alpha_{i_j}, (\underline{r}_{i_j} \circ \widehat{\pi}_{i_j})(y) \bigr \} \Bigr \},
\]
where $\alpha_{i_j}:=(\pi_{i_j} \circ G_{j-1} \circ T^m)(x) = (\pi_{i_j} \circ G_j \circ T^{m-1})(x)$. We further set 
\[
\beta_{i_j} := \left|\pi_{i_j} \Bigl( ( G_j \circ T^m)(x) - T^{m}(x) \Bigr) \right|
\]
for any $i_j \in I$ and observe that 
\begin{align*}
\beta_{i_j} 
&= \left|\pi_{i_j} \Bigl((G_j \circ T^{m})(x) - (G_j \circ T^{m-1})(x) \Bigr) \right| \\
&= \left|\pi_{i_j} \Bigl((G_j \circ T^{m})(x) - (\rho_{i_j} \circ G_j \circ T^{m-1})(x) \Bigr) \right| \\
&= \left|(f_{i_j} \circ G_{j-1} \circ T^m)(x) - (f_{i_j} \circ G_j \circ T^{m-1})(x) \right| \\
&\le \lambda \left\| ( G_{j-1} \circ T^{m})(x) - ( G_j \circ T^{m-1})(x) \right\|_{\infty} \\
&\le \lambda \left\|  ( G_{j-1} \circ T^m)(x) - (G_{j-1} \circ T^{m-1})(x)  \right\|_{\infty} \\
&\le \lambda \left\|  T^m(x) - T^{m-1}(x)  \right\|_{\infty}.
\end{align*}
Thus
\[
 \left\|  T^{m+1}(x) - T^{m}(x)  \right\|_{\infty} \le \max_{i_j \in I } \beta_{i_j} \le \lambda \left\|  T^{m}(x) - T^{m-1}(x)  \right\|_{\infty}.
\]
It easily follows that $(T^m(x))_{m \in \mathbb{N}}$ is a Cauchy sequence and thus converging to a fixed point $x^*$ of $T$. This implies in particular that $x^* \in Q$. 

We now prove the statement in case only $\mathfrak{R} \subset \mathrm{Lip}_1(l_\infty^{n-1}, \R)$ is assumed. Moreover, assume without loss of generality that $0 \in Q$. By Lemma \ref{Lem:l12}, it is enough to show that for any $R > 0$, the set $Q \cap B(0,R) \subset l_\infty^n$ is injective. Fix $R > 0$ and note that $g(0) = 0$ for any $g \in \mathfrak{R}$, hence $g(B(0,R)) \subset B(0,R)$ and thus  $-R \le (\underline{r}_i \circ \widehat{\pi}_i)(x) \le (\overline{r}_i \circ \widehat{\pi}_i)(x) \le R$ for any $x \in B(0,R)$. We can thus set for $k \in \mathbb{N}$ and $i \in I$: $\lambda_k := 1 - \frac{1}{k}$ as well as 
\[
(\underline{r}_i^k \circ \widehat{\pi}_i)(x) := \lambda_k \bigl[ (\underline{r}_i \circ \widehat{\pi}_i)(x) - R \bigr] + R
\]
and
\[ (\overline{r}_i^k \circ \widehat{\pi}_i)(x) := \lambda_k \bigl[ (\overline{r}_i \circ \widehat{\pi}_i)(x) - R \bigr] + R.
\]
Set now for any $k \in \mathbb{N}$:
\[
Q_k := \Bigl \{ x \in B(0,R) : \forall i \in I , \ (\underline{r}_i^k \circ \widehat{\pi}_i)(x) \le x_i \le (\overline{r}_i^k \circ \widehat{\pi}_i)(x) \Bigr \}.
\]
Note that $\mathfrak{R}^k := \{\underline{r}_i^k : i \in I \} \cup \{\overline{r}_{i}^k : i \in I \}$ satisfies $\mathfrak{R}^k \subset \mathrm{Lip}_{\lambda_k}(l_\infty^{n-1}, \R)$. Hence, we can apply the above argument and define the $1$-Lipschitz retraction $r^k \colon B(0,R) \to Q_k$ to be the pointwise limit of the sequence $(T^{m,k})_{m \in \mathbb{N}}$. It follows that $Q_k $ is injective. Finally, since the sequence $(Q_k)_{k \in \mathbb{N}}$ is decreasing for the inclusion and
\[
Q \cap B(0,R) = \bigcap_{k \in \mathbb{N}} Q_k,
\]
it follows that $Q \cap B(0,R)$ is injective (cf. for instance~\cite[Theorem~5.1]{EspK}).
\end{proof}

We shall later need a statement which is slightly more general than Proposition~\ref{Prop:p1} and whose proof is a direct analogue of the above proof. Let $I^1,I^2,I^3 \subset I_n$ with $I^i \cap I^j = \emptyset$ if $i \neq j$ and 
\[
\mathfrak{R}^1 := \{\underline{r}_{i} : i \in I\} \cup \{ \overline{r}_{i} : i \in I \}, \ \
\mathfrak{R}^2 := \{\underline{r}_{i} : i \in I^2\}, \ \
\mathfrak{R}^3 := \{\overline{r}_{i} : i \in I^3\}
\]
such that $\mathfrak{R}^1, \mathfrak{R}^2, \mathfrak{R}^3 \subset \mathrm{Lip}_1(l_\infty^{n-1}, \R)$. Set moreover
\begin{align*}
Q^1 &:= \Bigl \{ x \in \R^n : \forall i \in I^1 , \ (\underline{r}_i \circ \widehat{\pi}_i)(x)  \le x_i \le (\overline{r}_i \circ \widehat{\pi}_i)(x) \Bigr \},\\
Q^2 &:= \Bigl \{ x \in \R^n : \forall i \in I^2 , \ (\underline{r}_i \circ \widehat{\pi}_i)(x)  \le x_i \Bigr \}, \\
Q^3 &:= \Bigl \{ x \in \R^n : \forall i \in I^3 , \ x_i \le (\overline{r}_i \circ \widehat{\pi}_i)(x) \Bigr \},
\end{align*}
so that $Q^1,Q^2,Q^3 \subset l_{\infty}^n$. Assume finally that $Q := Q^1 \cap Q^2 \cap Q^3 \neq \emptyset$ and that for any $i \in I^1$, $\underline{r}_i \le \overline{r}_{i}$. It follows that $Q$ is injective.



\section{The Cone $\mathrm{K}_C$}\label{sec:s5}
For $j \in I_n=\{1,\dots,n\}$, let us define the cone
\[
C_j := \{ x \in \R^n : x_j = \left\| x \right\|_{\infty} \},
\]
note that 
\[
\mathrm{int}(C_j)= \{ x \in \R^n : x_j > \max_{i \in I_n \setminus \{ j \} } | x_i | \}
\]
and set 
\[
\mathcal{C} := \{-C_j : j \in I_n \} \cup \{C_j : j \in I_n \}.
\]
Let $\emptyset \neq C \subset l_{\infty}^n$ be a convex polyhedral cone; in particular, $0 \in \mathrm{apex}(C)$ and $C = C + C = \lambda C$ for $\lambda > 0$ . Define
\[
\mathcal{S}_C := \{ C' \in \mathcal{C} : \mathrm{int}(C') \cap C = \emptyset \}.
\]
Finally, set 
\begin{align*}
\bar{\mathrm{K}}_C 
:=& \ \{ p \in \R^n : \exists \  a \in \mathrm{apex}(C) \text{ such that } \{ C' \in \mathcal{C} : p \in a + C' \} \subset \mathcal{S}_C \} \\
=& \ \mathrm{apex}(C) + \left(\R^n \setminus \bigcup_{C' \in \mathcal{C} \setminus \mathcal{S}_C} C' \right)
\end{align*}
and
\begin{equation}\label{eq:e0000}
\mathrm{K}_C := \R^n \setminus \bar{\mathrm{K}}_C,
\end{equation}
noting in particular that $\mathrm{K}_C$ is a cone, $C \subset \mathrm{K}_C$ and $\mathrm{apex}(C) + \mathrm{K}_C = \mathrm{K}_C $. Although we shall use the above expression in the proof of Lemma \ref{Lem:l1}, note that $\mathrm{K}_C$ also admits the expression
\begin{equation}\label{eq:e0001}
\mathrm{K}_C := \bigcap_{a \in \mathrm{apex}(C)} \bigcup_{C' \in \mathcal{C} \setminus \mathcal{S}_C} (a + C').
\end{equation}
For a $\nu \in \R^n \setminus \{ 0\}$, let us denote by 
\[
H_{\nu} := \{ x \in \R^n : x \cdot \nu \ge 0 \}
\]
the corresponding inner half-space at the origin with normal vector $\nu$. Moreover, we shall again denote the standard basis of $\R^n$ by $\{e_1,\dots,e_n\}$. Note that in this notation and for any $j \in I_n$,
\[
C_j = \bigcap_{(i,\sigma) \in (I_n \setminus \{j\}) \times \{ \pm 1 \} } H_{e_j + \sigma e_i} \ \ \text{ and } \ \ -C_j = \bigcap_{(i,\sigma) \in (I_n \setminus \{j\}) \times \{ \pm 1 \} } H_{-e_j + \sigma e_i}.
\]
We shall now prove that the cone $\mathrm{K}_C \subset l_{\infty}^n$ is injective. The purpose of introducing $\mathrm{K}_C$ is that we shall be able to construct in the proof Theorem~\ref{Thm:t2} a $1$-Lipschitz retraction of $\mathrm{K}_C$ onto $C$. It will follow from Lemma~\ref{Lem:l22} that $\mathrm{K}_C$ consists of the union of $C$ and points $p \in l_{\infty}^n$ that are contained in a finite intersection $\bigcap_{i} B(x_i,r_i)$ of balls centered at points $x_i \in C$ such that 
\[
\mathrm{apex}(C) \cap \bigcap_{i} B(x_i,r_i)  = \emptyset.
\]


\begin{Lem}\label{Lem:l1}
Let $C \subset l_{\infty}^n$ be a convex polyhedral cone such that $\mathrm{int}(C) \neq \emptyset$, then $\mathrm{K}_C \subset l_{\infty}^n$ is injective.
\end{Lem}
\begin{proof}
We shall use Proposition \ref{Prop:p1}. We set
\begin{align*}
I^1 &:= \{ j \in I_n : \exists \sigma \in \{\pm 1\} \text{ such that } \sigma C_j \in \mathcal{S}_C \text{ and } -\sigma C_j \notin \mathcal{S}_C \}, \\
I^2 &:= \{ j \in I_n : \{C_j,-C_j\} \subset \mathcal{S}_C \}.
\end{align*}
Whenever $j \in I^1$ and $\tau C_j \in \mathcal{S}_C$, set
\[
I_{(j,\tau)}^1 := \{ (i,\sigma) \in (I_n \setminus \{j\}) \times \{ \pm 1 \}  : \sigma C_i \in \mathcal{S}_C \}
\]
and whenever $j \in I^2$, let
\[
I_{(j,\tau)}^2 := I_j^2 := \{ (i,\sigma) \in (I_n \setminus \{j\}) \times \{ \pm 1 \}  :  \{C_j,-C_j\} \subset \mathcal{S}_C \}.
\]
For $\alpha \in \{1,2\}$ and $j \in I^{\alpha}$, we define the cones
\[
\widetilde{C}_j^{\tau} := \bigcap_{(i,\sigma) \in [(I_n \setminus \{j\}) \times \{ \pm 1 \} ] \setminus  I_{(j,\tau)}^{\alpha}} H_{\tau e_j - \sigma e_i}
\]
with $\widetilde{C}_j^{\tau} := H_{ \tau e_j }$ if $[(I_n \setminus \{j\}) \times \{ \pm 1 \} ] \setminus  I_{(j,\tau)}^{\alpha} = \emptyset$ and define for $a \in \mathrm{apex}(C)$ and $x \in \R^n$ corresponding $1$-Lipschitz functions by
\[
r_{a}^{j,\tau}(x) := a_j + \tau \max_{(i,\sigma) \in [(I_n \setminus \{j\}) \times \{ \pm 1 \} ] \setminus  I_{(j,\tau)}^{\alpha}} \sigma (x_i - a_i).
\] 
If $\tau =1$, then $y \in a + \widetilde{C}_j^{1}$ if and only if $y_j \ge r_{a}^{j,1}(y)$. We set $(\overline{r}_j \circ \widehat{\pi}_j)(x) := \inf_{a \in \mathrm{apex}(C)} r_{a}^{j,1}(x)$ and
\[
N_{(j,1)}:= \bigcap_{a \in \mathrm{apex}(C)} \bigl[ \R^n \setminus \mathrm{int}(a + \widetilde{C}_j^{1}) \bigr] = \Bigl \{ x \in \R^n : x_j \le (\overline{r}_j \circ \widehat{\pi}_j)(x) \Bigr \}.
\]
If $\tau =-1$, then $y \in a + \widetilde{C}_j^{-1}$ if and only if $y_j \le r_{a}^{j,-1}(y)$. We  set $(\underline{r}_j \circ \widehat{\pi}_j)(x) := \sup_{a \in \mathrm{apex}(C)} r_{a}^{j,-1}(x)$ and
\[
N_{(j,-1)}:= \bigcap_{a \in \mathrm{apex}(C)} \bigl[ \R^n \setminus \mathrm{int}(a + \widetilde{C}_j^{-1}) \bigr] = \Bigl \{ x \in \R^n : x_j \ge (\underline{r}_j \circ \widehat{\pi}_j)(x) \Bigr \}.
\]
If $j \in I^1$ and $\tau C_j \in \mathcal{S}_C$, we set $N_{(j,-\tau)} := \R^n$. Now, if $j \in I^2$ we need to show that 
\begin{equation}\label{eq:e0}
\underline{r}_j \circ \widehat{\pi}_j \le \overline{r}_j \circ \widehat{\pi}_j,
\end{equation}
before we can apply the statement after Proposition \ref{Prop:p1}. Let us set 
\[
A_j := C_j \cup \bigcup_{ (l,\eta) \in I_{(j,1)}^2} \eta C_l.
\]
It is easy to see that $\mathrm{apex}(C) \cap \mathrm{int}(A_j) = \emptyset$ since $\mathrm{int}(C) \neq 0$. Furthermore, $\widetilde{C}_j^{1} \subset A_j$ since for $x \in  \widetilde{C}_j^{1}$, if $(i,\sigma) \in [(I_n \setminus \{j\}) \times \{ \pm 1 \} ] \setminus  I_{(j,1)}^2$, then $x_j \ge \sigma x_i$. Hence, either $x_j = \left\| x \right\|_{\infty}$ or there is $(l,\eta) \in I_{(j,1)}^2$ such that $\eta x_l = \left\| x \right\|_{\infty}$. It follows in particular that $\mathrm{apex}(C) \cap \mathrm{int}(\widetilde{C}_j^{1}) = \emptyset$. One then easily deduces (noting that $\widetilde{C}_j^{-1}=-\widetilde{C}_j^{1}$) that
\[
\bigl[ \mathrm{apex}(C) + \mathrm{int}(\widetilde{C}_j^{1}) \bigr] \cap \bigl[ \mathrm{apex}(C) + \widetilde{C}_j^{-1} \bigr] = \emptyset
\]
and this implies that $\underline{r}_j \circ \widehat{\pi}_j \le \overline{r}_j \circ \widehat{\pi}_j$. Indeed, if $r_{a}^{j,1}(y) < r_{a'}^{j,-1}(y)$ for some $y \in \R^n$, it follows that $[a + \mathrm{int}(\widetilde{C}_j^{1})] \cap [a' + \mathrm{int}(\widetilde{C}_j^{-1})] \neq \emptyset$.
Now, on the one hand, it is easy to see that setting 
\begin{equation}\label{eq:e1111111}
N_C := \bigcap_{(i,\sigma) \in (I^1 \cup I^2) \times \{\pm 1\}} N_{(i,\sigma)}
\end{equation}
it follows that $\mathrm{K}_C = N_C$. Indeed, note that $\R^n \setminus \mathrm{K}_C \subset \R^n \setminus N_C$ since if a face $F$ of $[-1,1]^n$ which satisfies 
\begin{align*}
F \in \mathcal{F} := 
 \bigl \{ &F' \in \mathrm{Faces}([-1,1]^n) \setminus \{[-1,1]^n\} : \\
& \forall (i,\sigma) \in I_n \times \{\pm 1\}, \text{ if } F' \subset \sigma C_i \text{ then } \sigma C_i \in \mathcal{S}_C \bigr \},
\end{align*}
then $\mathrm{relint}(F) \cap N_C = \emptyset$.  Indeed, in the asymmetric case where $F$ is such that $-F \notin \mathcal{F}$, there is then $ j \in I^1$ such that $F \subset \sigma C_j \subset \mathcal{S}_C$ for some $\sigma \in \{ \pm 1 \}$ and thus $\mathrm{relint}(F)$ is in the complement of $N_{(j,\sigma)}$. In the symmetric case where both $F$ and $-F$ are in $\mathcal{F}$, there is then $ j \in I^2$ such that $F \subset \sigma C_j$, $-F \subset -\sigma C_j$ and $\{C_j,-C_j\} \subset \mathcal{S}_C$ for some $\sigma \in \{ \pm 1 \}$, thus $\mathrm{relint}(F)$ is in the complement of $N_{(j,\sigma)}$. Hence in both cases and for any $\lambda > 0$, one has:
\[
\mathrm{relint}\Bigl(\lambda F + \mathrm{apex}(C) \Bigr) \cap N_C = \emptyset
\]
and thus $\R^n \setminus \mathrm{K}_C \subset \R^n \setminus N_C$. Now, note that if $x \in \R^n \setminus N_C$, then $x \in a + \mathrm{int}(\widetilde{C}_j^{\tau})$ for some $j \in I^{\alpha}$ verifying $\tau C_j \in \mathcal{S}_C$ for some $\tau \in \{ \pm 1\}$. Hence 
\[
x \in a + \mathrm{int}\left( \tau C_j \cup \bigcup_{ (l,\eta) \in I_{(j,\tau)}^{\alpha}} \eta C_l \right)
\]
and thus $x \notin \mathrm{K}_C$ by \eqref{eq:e0000}. Finally, by \eqref{eq:e0} we can, using \eqref{eq:e1111111}, apply the statement following the proof of Proposition \ref{Prop:p1} to $N_C \subset l_{\infty}^n$ in order to obtain that $N_C$ is injective and thus so is $\mathrm{K}_C \subset l_{\infty}^n$, which finishes the proof.
\end{proof}
To illustrate Lemma \ref{Lem:l1}, consider the case where $C = H_{e_n}$. It follows that $\mathrm{apex}(C)=\partial H_{e_n}$ and $\mathcal{S}_C=\{ -C_n \}$. Thus $\mathrm{K}_C = C =H_{e_n} $ is injective by Lemma~\ref{Lem:l1}, which we already know from the statement following the proof of Proposition~\ref{Prop:p1}. In the case where $C = C_n$, one has $\mathrm{apex}(C)=\{0\}$ and $\mathcal{S}_C= \mathcal{C} \setminus \{ C_n \}$. Hence $\mathrm{K}_C=C=C_n$ is injective as we already know. Finally, if 
\[
C = C_n^{\eps} := \left \{ x \in \R^n : x_n \ge \max_{i \in I_n \setminus \{ n \} } (1+\eps)| x_i | \right \} \subset \mathrm{int}(C_n) \cup \{ 0 \}
\]
for some $\eps >0$, then we again have $\mathrm{apex}(C)=\{0\}$, $\mathcal{S}_C=\mathcal{C} \setminus \{C_n \}$ and $\mathrm{K}_C=C_n$. We moreover denote by 
\[
\mathrm{aff}(X) := \left \{ \sum_{i=1}^{l} \alpha^i x^i : \{\alpha^1,\dots,\alpha^l\} \subset \R, \ \{x^1,\dots,x^l\} \subset X, \sum_{i=1}^{l} \alpha^i = 1 \right \} \subset \R^n
\]
the \textit{affine hull} of a subset $\emptyset \neq X \subset \R^n$. We now define a class of polytopes that can be obtained as a finite intersection of balls in $l_{\infty}^n$.

\begin{Def}\label{Def:d1}

For $1 \le k \le n$, let
\begin{align*}
\mathcal{I}_k :=  \Biggl \{ &[-1,1]^n \cap \bigcap_{j=1}^k \bigl( T_{p^j} F_j - p^j \bigr) \subset \R^n : \\
 &\text{ for all } j \in \{1,\dots,k\}, \  F_j \in \mathrm{Facets}([-1,1]^n) \text{ and } p^j \in F_j    \Biggr \}.
\end{align*}
\end{Def}

The next lemma will enable us to find for any $p \in \mathrm{K}_C \setminus C$ a face $F \in \mathrm{Faces}(P,\mathrm{apex}(C))^c$ of a polytope $P \in \mathcal{I}_k$ such that for some $\bar{\gamma} \in (0,\infty)$ and $\bar{a} \in \mathrm{apex}(C)$, $F_p:=\bar{\gamma}F + \bar{a}$ contains $p$. The interesting feature of $F_p$ will be that it is stable under any $1$-Lipschitz retraction of $l_{\infty}^n$ onto a set containing $C$. It is key that the set $\mathcal{I}_k$ is finite for every $k$ and that $\mathrm{Faces}(P)$ is finite for any $P \in \mathcal{I}_k$. 

\begin{Lem}\label{Lem:l22}
Let $C \subset l_{\infty}^n$ be a convex polyhedral cone such that $\mathrm{int}(C) \neq \emptyset$ and $0 \le k := \mathrm{dim}(\mathrm{apex}(C)) < n$.
Define $\Delta \colon \mathrm{apex}(C) \times (0,\infty) \to \R$ by
\[
\Delta(a,\gamma) := \min_{P \in \mathcal{I}_k} \min_{F \in \mathrm{Faces}(P,\mathrm{apex}(C))^c} d(\gamma F + a , \mathrm{apex}(C))
\]
with $\min_{F \in \mathrm{Faces}(P,\mathrm{apex}(C))^c} d(\gamma F + a, \mathrm{apex}(C)) := \infty$ if $\mathrm{Faces}(P,\mathrm{apex}(C))^c = \emptyset$. 
Then, for each $p \in \mathrm{K}_C$ so that $d(p,\mathrm{apex}(C)) = \eta >0$, there are $(\bar{a},\bar{\gamma}) \in \mathrm{apex}(C) \times [\eta,\infty)$, $P \in \mathcal{I}_k$ and $F \in \mathrm{Faces}(P,\mathrm{apex}(C))^c$ such that for $F_p := \bar{\gamma} F + a$, one has:
\begin{enumerate}[$(i)$]
\item $p \in F_p$ as well as
\item $d(F_p,\mathrm{apex}(C))$ is positive, $\Delta(0,1) \neq \infty$ and $\Delta(0,1)$ is positive as well. In addition: 
\[
d(F_p,\mathrm{apex}(C)) \ge \Delta(a,\bar{\gamma}) = \bar{\gamma} \Delta(0,1) \ge \eta \Delta(0,1).
\]
\item Moreover, for any set $C \subset X \subset l_{\infty}^n$ and any retraction $r \in \mathrm{Lip}_1(l_{\infty}^n, X)$ onto $X$, one has $r(F_p) \subset F_p$.
\end{enumerate}
\end{Lem}
In the proof of Lemma \ref{Lem:l22}, we shall, for given points $p$ and $q$ in $l_{\infty}^n$, consider
\[
\bigcup_{m \in \mathbb{N} \cap [n_0,\infty)} B(mq, \left\|mq - p \right\|_{\infty}) \subset l_{\infty}^n.
\]
It is not difficult to see that there is a threshold $n_0 \in \mathbb{N}$ as well as $\bar{q} \in \R q$ such that $\left\|mq - p \right\|_{\infty} = \left\|mq - \bar{q} \right\|_{\infty}$ for any $m \ge n_0$ and such that the sequence of balls $(B(mq, \left\|mq - p \right\|_{\infty}))_{m \in \mathbb{N} \cap [n_0,\infty)}$ is increasing. Altogether, this implies that the above union can be written as the tangent cone
\[
\mathrm{T}_{\bar{q}} B(n_0 q, \left\|n_0 q - p \right\|_{\infty} ).
\]
For a fixed point $p \in \mathrm{K}_C \setminus C$, we shall iterate in the proof below, the above observation as many times as the dimension $k$ of $\mathrm{apex}(C)$. Going from step $j$ to step $j+1$, we consider a particular increasing sequence of balls with centers on a line in $\mathrm{apex}(C)$ and whose union is the tangent cone $\mathrm{T}_{\bar{z}} B(z,R_1)$ as described above. Following an easy criterion described in the proof, we consider a corresponding sequence of balls centered on a ray in $C$ in the interior of a cone $\sigma_i C_{j_i} \in \mathcal{C}$ and once more, it follows as above, that their union can be written as a tangent cone to a ball, namely in this case $p + H _{\sigma_i e_{j_i}}$. These two tangent cones are defined in such a way that their intersection $G_{j+1}$ (which is then by definition an increasing union of intersection of balls centered in $C$) is $(n-1)$-dimensional. Hence, $\bigcap_{l=0}^{k} G_l$ is of a similar form and we shall show that 
\[
\mathrm{apex}(C) \cap \bigcap_{l=0}^{k} G_l = \{ \bar{a} \}.
\]
Finally, we shall consider the polytope $P := B(\bar{a},\bar{\gamma}) \cap \bigcap_{l=0}^{k} G_l$ which is a translated rescaling of a polytope in $\mathcal{I}_k$ (cf. Definition \ref{Def:d1}) and we shall show that $P$ has a face $p \in F_p $ which is disjoint from $\mathrm{apex}(C)$ and which can be written as a finite intersection of balls centered in $C$. In particular, $F_p$ is stable under any $1$-Lipschitz retraction of $l_{\infty}^n$ onto a subset containing $C$. Finally, note that if $C$ is injective (hence hyperconvex), then in particular $F_p \cap C \neq \emptyset$.

\begin{proof}[Proof of Lemma~\ref{Lem:l22}]
Fix $p \in \mathrm{K}_C$ such that $\eta := d(p,\mathrm{apex}(C)) > 0$. We set $A_0 := \mathrm{apex}(C)$, $G_0 := l_{\infty}^n$, $D_0 := l_{\infty}^n$. We continue inductively and define for $1 \le j+1 \le k$ the following
\[
A_{j+1} := \mathrm{apex}(C) \cap  \bigcap_{l=0}^{j+1} \mathrm{apex}(G_l) \ \  \text{ and } \ \ 
D_{j+1} := \bigcap_{l=0}^{j+1} \mathrm{aff}(G_l)
\]
as well as the sets $G_1,\dots,G_k$ along the following procedure: for each $0 \le j \le k-1$, choose arbitrarily  $a \in A_j$ and set $Y_j := B ( a, 1 ) \cap D_j$. Next, pick $q \in A_j$ such that the following hold:
\begin{enumerate}
\item If there is a facet $F$ of $Y_j$ such that $A_j \cap \mathrm{relint}(F) \neq \emptyset$, then $q \in A_j \cap \mathrm{relint}(F)$.
\item If for any facet $F'$ of $Y_j$, one has $A_j \cap \mathrm{relint}(F') = \emptyset$, then there is a face $F''$ of $Y_j$ with $\mathrm{dim}(F'') \le \mathrm{dim}(Y_j) - 2$ such that $A_j \subset \mathrm{aff}(F'' \cup \{a\})$ and then $q \in A_j \cap \mathrm{relint}(F'')$.
\end{enumerate}
It is not difficult to see that exactly one of these two cases occur. Let us now set $q^{m} := a + m (q - a)$ for $m \in \mathbb{N}$. There exists $m_1 > 0$ such that one can find $\mathfrak{I} := \{(j_1,\sigma_1),\dots,(j_N,\sigma_N)\} \subset I_n \times \{\pm 1\}$ so that $\left\|p - q^{m} \right\|_{\infty} = \sigma_i (p_{j_i} - q_{j_i}^{m})$ for $m \ge m_1$ if and only if $(j_i,\sigma_i) \in \mathfrak{I}$. Hence $p \in  q^{m} + [\bigcap_{(j_i,\sigma_i) \in \mathfrak{I} } \sigma_i C_{j_i} \setminus \bigcup_{(l,\tau) \notin \mathfrak{I}} \tau C_{l}] $. Since $p \in \mathrm{K}_C$ and $q^{m} \in \mathrm{apex}(C)$, it follows that there is some $(j_i,\sigma_i) \in \mathfrak{I}$ such that $w \in C \cap \mathrm{int}(\sigma_i C_{j_i}) \neq \emptyset$. We then set $w^{m} := a + m w \in C \cap [ a + \mathrm{int}(\sigma_i C_{j_i}) ]$. As we noted before the proof, one can find $z,\bar{z} \in a + \R (q - a)$, as well as $R_1>0$ and $m_2 \in \mathbb{N} \cap [m_1,\infty)$ such that:
\[
\mathrm{T}_{\bar{z}} B(z,R_1) = \bigcup_{m \ge m_2} B ( q^{m},  \left\|q^{m} - p \right\|_{\infty})
\]
and $v,\bar{v} \in a + \R w$ as well as $R_2>0$ such that
\[
p + H _{\sigma_i e_{j_i}} = \mathrm{T}_{\bar{v}} B(v,R_2) = \bigcup_{m \ge m_2} B ( w^{m},  \left\|w^{m} - p \right\|_{\infty}).
\]
We then set
\[
G_{j+1} := \mathrm{T}_{\bar{z}} B(z,R_1) \cap (p + H _{\sigma_i e_{j_i}})
\]
which is a face of $\mathrm{T}_{\bar{z}} B(z,R_1)$ and thus in particular a cone with 
\[
\mathrm{apex}(G_{j+1}) = \mathrm{apex}(\mathrm{T}_{\bar{z}} B(z,R_1)).
\]
By construction, we can define the re-indexing $1 \le f(j+1) := j_i \le n$ such that 
\[
\mathrm{aff}(G_{j+1}) = p + \partial H _{e_{f(j+1)}}.
\]
There is $I(j):= \{f(1),\dots,f(j)\} \subset I_n$ such that for any $x,y \in D_j$ and for any $f(l) \in I(j)$, $x_{f(l)} = y_{f(l)}$. Therefore, since for $m \ge m_2$ both $p$ and $q^{m}$ are in $D_j$ and $p \in  q^{m} + \sigma_i C_{j_i}$ it follows in particular that $j_i \notin I(j)$. Hence $q^{m} \notin \mathrm{aff}(G_{j+1})=p + \partial H _{e_{j_i}}$ and therefore $\emptyset \neq \mathrm{aff}(G_{j+1}) \cap A_j \neq A_j$. Now, it is easy to see that for $1 \le j+1 \le k$, one has:
\[
G_{j+1} \cap A_{j} = \mathrm{apex}(G_{j+1}) \cap A_{j} = \mathrm{aff}(G_{j+1}) \cap A_{j},
\]
$\mathrm{dim}(A_{j+1}) = \mathrm{dim}(A_{j})-1$ and $A_k=\{\bar{a}\} \subset \mathrm{apex}(C)$. We finally set $\bar{\gamma} := \left\| \bar{a}- p \right\|_{\infty} \ge \eta$ and 
\[
P := B(\bar{a},\bar{\gamma}) \cap \bigcap_{l=0}^{k} G_l.
\]
Similarly to what we have argued before, since $p \in \mathrm{K}_C$ there is $b \in C \cap \mathrm{int}(\bar{a} + \tau C_{n_0})$ where $n_0 \notin I(k)$ such that setting $\beta := \left\|b - p \right\|_{\infty}$ and
\[
Q := B(\bar{a},\bar{\gamma}) \cap D_k,
\]
one has that $\bar{F} := B(b, \beta) \cap Q$ is a facet of $Q$ in $D_k$. Setting finally $F_p := \bar{F} \cap P = B(b, \beta) \cap P$, it follows that $F_p$ has the desired properties, in particular it is a face of $P$ (Remark that $F_p = \bar{F} \cap P = (\mathrm{aff}(\bar{F}) \cap Q ) \cap P = \mathrm{aff}(\bar{F}) \cap P$ and there is a half-space $H$ of $D_{k}$ such that $\mathrm{rel} \partial H = \mathrm{aff}(\bar{F})$ and $P \subset Q \subset H$. Hence $F_p$ is a face of $P$ cf. \cite[Chapter~2]{Zie}) and note that $P$ is a translated rescaling (with parameters $\bar{a}$ and $\bar{\gamma}$) of a polytope in $\mathcal{I}_k$. This proves $(i)$.

Moreover, $d(F_p,\mathrm{apex}(C))$ is positive since 
\[
F_p \cap \mathrm{apex}(C) = F_p \cap D_k \cap \mathrm{apex}(C) = F_p \cap \{ \bar{a} \} = \emptyset.
\]
The rest of $(ii)$ is easily seen to hold. Indeed, $\Delta(0,1)$ is positive since $\mathcal{I}_k$ is a finite set and thus up to rescaling and translation along points of $\mathrm{apex}(C)$, there are only finitely many different intersections of a hyperplane of the form $p + H _{\sigma_i e_{j_i}}$ with a tangent cone to a ball like $\mathrm{T}_{\bar{z}} B(z,R_1)$ and thus there are only finitely many different outcomes for the sets $G_1,\dots,G_k$ depending only on the dimension of $l_{\infty}^n$ and independently of the particular $C$.

Since $P$ is bounded and looking at the definition of the sets $G_1,\dots,G_k$; it is clear that the set $P$ can be expressed as an intersection of closed balls centered in $C$ that are pairwise intersecting and note that such balls are stable under $r$ as given in $(iii)$. This finally concludes the proof of the Lemma.
\end{proof}

To illustrate Lemma \ref{Lem:l22}, consider again the case where 
\[
C = C_n^{\eps} := \{ x \in \R^n : x_n \ge \max_{i \in I_n \setminus \{ n \} } (1+\eps)| x_i | \} \subset \mathrm{int}(C_n) \cup \{ 0 \}
\]
for some $\eps >0$ and consequently $\mathrm{apex}(C)=\{0\}$, $\mathcal{S}_C=\mathcal{C} \setminus \{C_n \}$ and $\mathrm{K}_C=C_n$. For any $p = (p_1,\dots,p_n) \in C_n \setminus C_n^{\eps}$, one has 
\[
F_p=B(0,\left\|p \right\|_{\infty})\cap (p+H_{e_n})= \{x \in l_{\infty}^n : \left\|x \right\|_{\infty} = \left\|p \right\|_{\infty} \text{ and } x_n = p_n \}.
\]
Now, in the case $C = C_n^{\eps} + \R e_{n-1}$, one has $\mathrm{apex}(C) = \R e_{n-1}$, $\mathrm{K}_C = C_n + \R e_{n-1}$. For any $p = (p_1,\dots,p_n) \in \mathrm{K}_C \setminus C$, one has with $\bar{p} := (0,\dots,0,p_{n-1},0)$:
\begin{align*}
F_p &= \partial(p+H_{e_{n-1}}) \cap B(\bar{p},p_n) \cap (p+H_{e_n}) \\
&= \{x \in l_{\infty}^n : \left\|x - \bar{p} \right\|_{\infty} = p_n, \ x_{n-1} = p_{n-1} \text{ and } x_n = p_n \}.
\end{align*}
\section{Injective Convex Polyhedral Cones}
We shall also make use in the next lemma of the observation we made before the proof of Lemma \ref{Lem:l22}.
\begin{Lem}\label{Lem:l3}
Let $C \subset l_{\infty}^n$ be an injective convex polyhedral cone with non-empty interior such that 
for any $F \in \mathrm{Facets}^*([-1,1]^{n},C)$, $-F \in \mathrm{Facets}^*([-1,1]^{n},C)$ as well. Then $C = \R^n$.
\end{Lem}
\begin{proof}
By assumption there is a subset $I \subset I_n=\{1,\dots,n\}$ such that
\begin{equation}\label{eq:e1}
C \cap  \mathrm{int}(\sigma C_j) \neq \emptyset \ \text{ if and only if } \ (j,\sigma) \in I \times \{ \pm 1 \}.
\end{equation}
Let us assume for simplicity that $I = \{1,\dots,k\}$ with $I := \emptyset$ if $k=0$. Note that by $\eqref{eq:e1}$, for any $i \in I$ and any $x \in \R^n$ there is $(u^i,v^i) \in [\mathrm{int}(C) \cap \mathrm{int}(C_i)] \times [\mathrm{int}(C) \cap \mathrm{int}(-C_i)]$ such that $m u^i + x \in \mathrm{int}(C) \cap \mathrm{int}(C_i)$ and $m v^i + x \in \mathrm{int}(C) \cap \mathrm{int}(-C_i)$ for any $m \in \mathbb{N}$ as well as:
\[
x + \partial H_{e_{i}} = \bigcup_{m \in \mathbb{N}}  B(m u^i + x ,\left\| m u^i \right\|_{\infty}) \cap \bigcup_{m \in \mathbb{N}}  B(m v^i + x ,\left\| m v^i \right\|_{\infty})
\]
(where $H_{\nu}$ is defined in Section \ref{sec:s5}). Setting $ U_m^i + x := B(m u^i + x ,\left\| m u^i \right\|_{\infty})$ and $ V_m^i + x :=  B(m v^i + x ,\left\| m v^i \right\|_{\infty})$, we obtain
\[
\bigcap_{i \in I} (x + \partial H_{e_{i}}) = \bigcap_{i \in I} \left( \bigcup_{m \in \mathbb{N}} [U_m^i + x] \cap \bigcup_{m \in \mathbb{N}} [V_m^i + x] \right).
\]
It follows that there are $m_1,\dots,m_k, n_1,\dots,n_k \in \mathbb{N}$ such that
\[
x \in \bigcap_{i \in I} \left(  [U_{m_i}^i + x] \cap [V_{n_i}^i + x] \right) =: S \subset \bigcap_{i \in I} (x + \partial H_{e_{i}}).
\]
Note that $S$ is an intersection of closed balls with centers in $C$ and pairwise intersecting in $l_{\infty}^n$ (since they all contain $x$), hence $S \cap C \neq \emptyset$ by hyperconvexity of $C$. We then deduce
\begin{equation}\label{eq:e2}
\Bigl( \{ x_1 \} \times \cdots \times \{ x_k \} \times \R^{n-k} \Bigr) \cap C \neq \emptyset
\end{equation}
for any $\{x_i\}_{i \in I} \subset \R$. Set 
\[
\pi := \widehat{\pi}_{k+1} \circ \cdots \circ \widehat{\pi}_{n}
\]
with $\pi \equiv 0$ if $k = 0$ and $\pi := \mathrm{id}_{\R^n}$ if $k = n$. From \eqref{eq:e2}, it follows $\pi(C) = \R^k$. Assume now by contradiction that $\pi(\mathrm{apex}(C)) \neq  \R^k$.  Pick $p\in C$ such that $\pi(p) \notin \pi(\mathrm{apex}(C))$ and pick $q \in C \cap \pi^{-1}(\{-\pi(p)\})$. Remark that setting $z := q + p \in C \setminus \mathrm{apex}(C)$ one has $z \neq 0$ and $\pi(z) = 0$. Hence $\max_{1 \le j \le k} |z_j|=0 <  \max_{k+1 \le j \le n} |z_j|$ thus $\max_{1 \le j \le k} |z_j| < \left\|z \right\|_{\infty}$ and therefore $z \notin \cup_{1 \le j \le k} \left[ C_j \cup (-C_j) \right]$. Since $\mathrm{int}(C) \neq \emptyset$ it follows that $C \cap \mathrm{int}(\sigma C_l) \neq \emptyset$ for some $(l,\sigma) \notin I \times \{ \pm 1 \}$ which contradicts \eqref{eq:e1}. Thus $\pi(\mathrm{apex}(C)) =  \R^k$. Hence, for any $y \in \R^k$, there is $w \in \mathrm{apex}(C)$ such that $\pi(w) = y$. Assume now  by contradiction that there is $w' \in C$ such that $\pi(w') = y$ and $w' \neq w$. Then $z := w' - w \in C \setminus \{ 0\}$ satisfies $\pi(z) = 0$ thus $\max_{1 \le j \le k} |z_j| = 0 < \max_{k+1 \le j \le n} |z_j|$ and this as before contradicts \eqref{eq:e1}. It follows that $\pi \colon C \rightarrow \R^k$ is injective. By definition of $\pi$ and since $\mathrm{int}(C) \neq \emptyset$, we deduce that $k = n$ thus $C = \R^n$. This proves the Lemma.
\end{proof}

Let $d_H(A,B)$ denote the Hausdorff distance of two subsets $\emptyset \neq A,B \subset l_{\infty}^n$, in other words
\[
d_H(A,B) := \inf \{ r \in (0,\infty) : A \subset N(B,r) \text{ and } B \subset N(A,r) \} \in [0,\infty]
\]
with $\inf \emptyset := \infty$.

The strategy to show (in the proof of Theorem~\ref{Thm:t2}) that $(i)$ and $(ii)$ imply the injectivity of $C$ is to construct a $1$-Lipschitz retraction $r$ of $\mathrm{K}_C$ onto $C$. In order to do so, we shall consider an increasing sequence $(l \alpha q+C)_{l \in \mathbb{N}}$ of translates of $C$ along $\R q$ with $\alpha>0$. The direction $q$ is chosen such that $-q \in \mathrm{int}(C)$, in order that $C \subset l \alpha q +C$ and $\cup_{l \in \mathbb{N}}(l\alpha q+C)=\R^n$. Moreover, $q$ is chosen so that for a facet $F$ of $[-1,1]^n=B(0,1)\subset l_{\infty}^n$ such that $F \notin \mathrm{Facets}^*([-1,1]^{n},C)$ and $-F \in \mathrm{Facets}^*([-1,1]^{n},C)$, one has $q \in \mathrm{relint}(F)$ which implies $d(q+\mathrm{apex}(C),\mathrm{K}_C) > 0$. We shall define $r$ as the composition $r_2 \circ r_1$ of two $1$-Lipschitz retractions. The points of $\mathrm{K}_C \setminus C$ that have distance to $\mathrm{apex}(C)$ greater than a fixed constant will be mapped by $r_1$ to $C$. The purpose of $r_2$ is then to map the points situated in a neighborhood of $\mathrm{apex}(C)$ but which are outside $\mathrm{apex}(C)$, onto $C$. 

Starting with the definition of $r_1$, we shall let $r^l$ be the composition of retractions onto the tangent cones of $l\alpha q+C$ that are different from $l\alpha q+C$ itself and we shall let $r_1$ be the inverse limit of the system $(r^{l})_{l \in \mathbb{N}}$, similarly to the proof of Theorem~\ref{Thm:t00001}. After that, we shall define $r_2$ as the pointwise limit of the composition of a system of $1$-Lipschitz retractions $(\rho^{k})_{k \in \mathbb{N}}$. The map $\rho^{k}$ will be the composition of a fixed number of $1$-Lipschitz retractions $\rho^{k,l}$ defined (similarly as $r^l$ above) as the composition of retractions onto the tangent cones of $l \frac{\alpha q}{2^k}+C$ (different from $l \frac{\alpha q}{2^k}+C$ itself).  

To prove that $r := r_2 \circ r_1$ is the desired map,
we shall note that the $1$-Lipschitz retractions used to define $r$ are all $1$-Lipschitz retractions of $l_{\infty}^n$ onto a set containing $C$. Lemma~\ref{Lem:l22} provides for any $p \in \mathrm{K}_C \setminus C$ a polytope $F_p$ containing $p$, stable under $r$ and such that $F_p \cap \mathrm{apex}(C) = \emptyset$. In particular, $r$ induces a $1$-Lipschitz retraction of $F_p$ onto $F_p \cap C$. To show that the image of $r$ is exactly $C$, we shall consider in a particular neighborhood of $\mathrm{apex}(C)$, an arbitrary point $p \in C_{k,l_0+1} \cap (\mathrm{K}_C \setminus C_{k,l_0})$ where $C_{k,l_0} = l_0 \frac{\alpha q}{2^k}+C$ and consider the map $\rho^{k,l_0}$ which consists of the composition of every $1$-Lipschitz retraction onto the tangent cones of $C_{k,l_0}$ (different from $C_{k,l_0}$ itself). We shall show that there is a ball $U(p_0,\delta_{p_0})$ containing $p$ and centered in $C_{k,l_0}$ such that
\[
U(p_0,\delta_{p_0} ) \cap C_{k,l_0} = U (p_0,\delta_{p_0} ) \cap \mathrm{T}_{p_0} C_{k,l_0}.
\]
This step is similar to an argument in the proof of Theorem~\ref{Thm:t00001} with the key difference that it is here important that $p_0 \notin \mathrm{apex}(C_{k,l_0})$, in order that $ C_{k,l_0} \subsetneq \mathrm{T}_{p_0} C_{k,l_0}$ and by definition of $\rho^{k,l_0}$ that consequently $\rho^{k,l_0}(p) \subset C_{k,l_0}$. We can repeat this procedure until $l_0=0$ to obtain $\rho^{k}(p) \in C$. 

We shall use indifferently the notation $[-r,r]^{n}$ and $B(0,r)$ in the following proof since both denote the same subset of $l_{\infty}^n$.

\begin{proof}[Proof of Theorem~\ref{Thm:t2}]
If $C$ is injective, we know by Theorem \ref{Thm:t00001} that its tangent cones are all injective. Furthermore, $(ii)$ follows from Lemma \ref{Lem:l3}.

Assume now that $(i)$ and $(ii)$ hold. Pick a facet $F$ of $[-1,1]^{n}$ such that $F \notin \mathrm{Facets}^*([-1,1]^{n},C)$ as well as $-F \in \mathrm{Facets}^*([-1,1]^{n},C)$ and pick $q \in \mathrm{relint}(F)$ such that $-q \in \mathrm{relint}(-F) \cap \mathrm{int}(C)$. Remark that 
\[
\mathrm{int}([0,\infty)F) + \mathrm{apex}(C)  \subset \R^n \setminus \mathrm{K}_C.
\]
For $R > 0$, set 
\[
\Sigma_R:= \mathrm{K}_C \cap \bigl[ B(0,R) + \mathrm{apex}(C) + [0,\infty)q \bigr]^c.
\]
Let us define the map $\bar{\Delta} \colon \mathrm{apex}(C) \times (0,\infty) \to \R$ by
\begin{equation}\label{eq:e6}
\bar{\Delta}(a,\gamma) := \min_{P \in \mathcal{I}_k} \min_{F' \in \mathrm{Faces}(P,\mathrm{apex}(C))^c} d((\gamma F' + a)  \cap \mathrm{K}_C, [0,\infty)q + \mathrm{apex}(C)).
\end{equation}
where $k := \mathrm{dim}(\mathrm{apex}(C))$. It is easy to see with the help of Lemma \ref{Lem:l22} that $\eps := \bar{\Delta}(0,1) > 0$ and thus by rescaling
\begin{equation}\label{eq:e3}
\bar{\Delta}(a,\kappa) = \bar{\Delta}(0,\kappa) = \kappa \bar{\Delta}(0,1) = \kappa \eps.
\end{equation}
Furthermore, there is $\bar{\eps} \in (0,\eps)$ such that $C \cup \bigcup_{p \in \partial C} U(p,\bar{\eps}) \subset C \cup \bigcup_{p \in \partial C} U(p,\eps_p)$ where for any $p \in \partial C$, we set
\begin{equation}\label{eq:e33}
\eps_p := \sup \{  \delta \in (0,\eps): U(p,\delta) \cap \mathrm{T}_pC = U(p,\delta) \cap C \},
\end{equation}
cf. proof of Theorem \ref{Thm:t00001}. Let us then choose $\alpha \in [0,\infty)$ such that
\begin{equation}\label{eq:e4}
d_H(C,\alpha q+C) < \bar{\eps}/2.
\end{equation}
Since by definition, one has $[0,\infty)q + C = l_{\infty}^n$, there is $m \in \mathbb{N}$ so that
\[
B(0,1) + \mathrm{apex}(C) \subset m \alpha q + C
\]
which after rescaling becomes 
\begin{equation}\label{eq:e5}
B(0,1/2^k) + \mathrm{apex}(C) \subset \frac{m \alpha q}{2^k} + C.
\end{equation}
Let $\{T_j\}_{j \in \{1,\dots,N\}}$ be an enumeration of the set:
\[
\bigl \{ \mathrm{T}_{p}(C) : \text{there is } F \in \mathrm{Faces}(C) \setminus \{ \mathrm{apex}(C)\} \text{ such that } p \in \mathrm{relint}(F) \bigr \}.
\]
If for each $j \in \{1,\dots,N\}$, we pick a $1$-Lipschitz retraction $\rho_j \colon l_{\infty}^n \to T_j$, then $\rho := \rho_N \circ \cdots \circ \rho_1$ defines a $1$-Lipschitz retraction of $\alpha q+C$ onto $C$, cf. proof of Theorem \ref{Thm:t00001}. Let us now for $y \in X$ denote by $\tau_y$ the translation map $ x \mapsto x + y$. For $l \in \mathbb{N}$, the map 
\[
r^{l} := \tau_{l \alpha q} \circ \rho \circ \tau_{-l \alpha q}
\]
is a $1$-Lipschitz retraction of $(l+1) \alpha q+C$ onto $l \alpha q+C$. We then define
\[
r_1(x) := (r^0 \circ r^1 \circ \cdots \circ  r^M)(x) 
\]
where $M$ is the smallest natural such that $x \in M\alpha q + C$. Similarly, for any $j \in \{1,\dots,N\}$, $k \in \mathbb{N} \cup \{0 \}$ and $l \in \{0,\dots,m\}$, we set 
\[
\rho^{k,l} : = \tau_{l \alpha q / 2^{k}  } \circ \rho \circ \tau_{-l \alpha q/ 2^{k}}
\]
as well as
\[
\rho^{k} : = \rho^{k,0} \circ \cdots \circ \rho^{k,m}.
\]
We then define
\[
r := r_2 \circ r_1
\]
by setting for any $y \in r_1(\mathrm{K}_C)$:
\[
r_2(y) := \lim_{k \rightarrow \infty} (\rho^k \circ \rho^{k-1} \cdots \circ \rho^1)(y).
\]

We shall now show that $r$ is well-defined, $r|_{C} = \mathrm{id}_C$ and $r \in \mathrm{Lip}_1(\mathrm{K}_C,C)$. This implies that $C$ is injective by Lemma \ref{Lem:l1}. Consider first $R \in (1/2^{k+1},1/2^{k}]$ with $k \in \mathbb{N} \cup \{0 \}$ and let $p \in \mathrm{K}_C$ be a point at distance $R$ from $\mathrm{apex}(C)$. Borrowing its notation, we can by Lemma \ref{Lem:l22} find a corresponding $F_p$ containing $p$ such that by \eqref{eq:e6} and \eqref{eq:e3}, one has
\begin{equation}\label{eq:e111}
F_p \cap \mathrm{K}_C \subset \Sigma_{\frac{\eps}{2^{k+1}}}.
\end{equation}
Assume that $p \notin C$. Note that by $(iii)$ in Lemma \ref{Lem:l22} and since it is easy to see that $r(\mathrm{K}_C) \subset \mathrm{K}_C$, one has
\[
r(F_p \cap \mathrm{K}_C) \subset F_p \cap \mathrm{K}_C.
\]
Let us set $C_{k,l} := l \frac{\alpha q}{2^k} + C$ for any $l \in \{0, \dots, m\}$. By \eqref{eq:e5}, there is then $l_0 \in \{0, \dots, m-1\}$ such that $p \in C_{k,l_0+1} \setminus C_{k,l_0}$ since $p$ was chosen so that
\[
p \in \partial [B(0,R) + \mathrm{apex}(C)] \cap \mathrm{K}_C 
\subset B(0,1/2^{k}) + \mathrm{apex}(C).
\]
It follows by \eqref{eq:e4} that
\[
d_H(C_{k,l_0},C_{k,l_0+1})= d_H(C,C_{k,1}) = \frac{1}{2^{k}} d_H(C,\alpha q+C) < \frac{\bar{\eps}}{2^{k+1}}.
\]
Therefore, noting that if $z \in C_{k,l_0}$ then $\sigma(z) := 2^{k} \left( z - l_0 \frac{\alpha q}{2^{k}} \right)  \in C$, one sees (cf. \eqref{eq:e33} for the definition of $\eps_{\sigma(z)}$) that 
\[
p \in \bigcup_{z \in \partial C_{k,l_0}} U (z,\bar{\eps}/2^{k+1}) 
\subset \bigcup_{z \in \partial C_{k,l_0}} U (z, \eps_{\sigma(z)}/2^{k} ).
\]
Hence, there is $p_0 \in \partial C_{k,l_0}$ such that $p \in U(p_0,\delta_{p_0})$, $\delta_{p_0} < \frac{\bar{\eps}}{2^{k+1}}$ and
\[
U ( p_0,\delta_{p_0} ) \cap C_{k,l_0} = U (p_0,\delta_{p_0} ) \cap \mathrm{T}_{p_0} C_{k,l_0}.
\]
From $\delta_{p_0} < \frac{\bar{\eps}}{2^{k+1}}$ and $\bar{\eps} < \eps$, it follows that $p_0 \notin \mathrm{apex}(C_{k,l_0})$ because by \eqref{eq:e111}:
\[
d(F_p \cap \mathrm{K}_C, \mathrm{apex}(C_{k,l_0})) \ge
d(F_p \cap \mathrm{K}_C, [0,\infty)q + \mathrm{apex}(C)) \ge 
\frac{\eps}{2^{k+1}} >
\delta_{p_0}.
\]
There is then $j \in \{1, \dots, N\}$ such that $\mathrm{T}_{p_0} C_{k,l_0} = l_0 \frac{\alpha q}{2^k} + T_j$. Hence $\rho^{k,l_0}(p) \in C_{k,l_0}$ and thus $\rho^k(p) \in C$.

The case where $p \in \mathrm{K}_C$ is a point at distance $R \ge 1$ from $\mathrm{apex}(C)$ is similar. It follows that $r$ is well-defined and it is then obviously a $1$-Lipschitz retraction onto $C$. This finally concludes the proof.
\end{proof}


\section{Graph Representation of Linear Systems of Inequalities with at most Two Variables per Inequality}

Let $\emptyset \neq Q \subset \R^n$ be an intersection of general half-spaces, that is half-spaces that are either closed or open. To a general half-space $H$ containing $Q$, we assign its inner normal vector $\nu \in \R \setminus\{ 0 \}$ in order that there is $p \in \R^n$ such that $H = p + H_\nu$ if $H$ is closed and $H = p + \mathrm{int}(H_\nu)$ if $H$ is open (recalling that $H_{\nu} := \{ x \in \R^n : x \cdot \nu \ge 0 \}$). For $n \in \mathbb{N}$, let us denote by $\mathcal{Z}_n$ the family of every $Q$ so that there is a set $\mathcal{N}(Q) \subset \R \setminus\{ 0 \}$ such that the following hold:
\begin{enumerate}[$(a)$]
\item $\mathcal{N}(Q)$ is finite and $Q$ can be written as the intersection over all $\nu \in \mathcal{N}(Q)$ of a general half-space with inner normal vector $\nu$.
\item For every $\nu \in \mathcal{N}(Q)$, there exist $f_{\nu},g_{\nu} \in \{0\} \cup \{e_1,\dots,e_n\}$ and $a_{\nu},b_{\nu} \in \R$ so that $f_{\nu} \neq g_{\nu}$ as well as
\[
\nu = a_{\nu} f_{\nu} + b_{\nu} g_{\nu}.
\]
\end{enumerate}
We now describe a construction that is introduced in \cite{Sho}. Every $Q \in \mathcal{Z}_n$ is the solution set of a linear system of inequalities of the form 
\[
\Sigma := \{ a_\nu y_{\nu} + b_{\nu} z_{\nu}  \succeq c_{\nu} \}_{\nu \in \mathcal{N}(Q)}
\]
where $\succeq$ stands for $\ge$ in some inequalities and possibly for $>$ in some others and $y_{\nu},z_{\nu} \in \{ x_0,x_1,\dots,x_n\}$ denote variables so that $y_{\nu} = x_i$ if $f_{\nu} = e_i$ as well as $z_{\nu} = x_j$ if $g_{\nu} = e_j$ and $y_{\nu} = x_0$ if $f_{\nu} = 0$. Conversely, to any system of linear inequalities as above, we can associate an element of $\mathcal{Z}_n$. Now, we can require all variables appearing in $\Sigma$ to have nonzero coefficients except the \textit{zero variable} $x_0$ which we additionally require to appear only with coefficient zero. We can associate to $\Sigma$ an undirected labeled multigraph without self-loops $\Gamma_{\Sigma} := (\mathrm{V}_{\Sigma},\mathrm{E}_{\Sigma})$ where the vertex set $\mathrm{V}_{\Sigma}$ is given by $\{x_0,x_1,\dots,x_n\}$ and the set $\mathrm{E}_{\Sigma}:=\{ E_{\nu} \}_{\nu \in \mathcal{N}(Q) } $ consists of all the labeled edges $E_{\nu} = \bigl( \{ y_{\nu}, z_{\nu} \},\Sigma_{\nu} \bigr) $ where $\Sigma_{\nu}$ denotes the inequality $ a_\nu y_{\nu} + b_{\nu} z_{\nu}  \succeq c_{\nu} $.
Note that $\Gamma_{\Sigma}$ does not contain any self-loop since we require $y_{\nu} \neq z_{\nu}$, that is all equations in $\Sigma$ contain two different variables. Equations that contain only one variable different from $x_0$ are given by edges connecting to $x_0$ and remark that $\Sigma$ does not contain any trivial inequalities like $1\ge 0$ or $-1/3>0$. A \textit{path} $P$ in $\Gamma_{\Sigma}$ is then given by 
\begin{equation}\label{eq:equation1}
\bigl( (v_1,\dots,v_{m+1}), E_1,\dots,E_m \bigr )
\end{equation}
where $ (v_1,\dots,v_{m+1})$ is a sequence of vertices in $\mathrm{V}_{\Sigma}$ and $(E_1,\dots,E_m)$ a sequence of labeled edges in $\mathrm{E}_{\Sigma}$ such that for each $l \in \{1,\dots,m\}$, one has:
\[
E_l= \left( \{ v_l,v_{l+1} \} ,  a_l v_l + b_l v_{l+1}  \succeq c_l  \right).
\]
We call $P$ \textit{admissible} if for each $l \in \{1,\dots,m-1\}$, the coefficients $b_l$ and $a_{l+1}$ have opposite signs (i.e., one is strictly positive and the other one is strictly negative). Note that if $P$ is admissible, one has $v_l \neq x_0$ for each $l \in \{2,\dots,m-1\}$ because we have required that $x_0$ appears only with zero coefficient. Admissible paths correspond to sequences of inequalities that form transitivity chains, the three inequalities $ 2 x_1 - 3 x_2 > - 4$, $ 2 x_2 + x_3 \ge 4 $ and $-x_3 -x_1 \ge 0$ give e.g. rise to an admissible path. However, the three inequalities $x_1 - x_2 \ge 0$, $x_2 - x_3 \ge 0$ and $-x_3 - x_4 \ge 0$ cannot label an admissible path since the coefficients of $x_3$ have the wrong relative signs. A path is called a \textit{loop} if its first and last vertices are identical and a loop is said to be \textit{simple} as soon as its intermediate vertices are distinct. The reverse of an admissible loop is admissible and cyclic permutations of a loop $P$ given by \eqref{eq:equation1} are admissible if and only if $a_1$ and $b_m$ have opposite signs, in which case $P$ is called \textit{permutable}. Note also that since $x_0$ appears in $\Sigma$ only with zero coefficient, no admissible loop with initial vertex $x_0$ is permutable.

For an admissible path $P$ given again by \eqref{eq:equation1}, let us define the \textit{residue inequality} of $P$ to be the inequality obtained by applying transitivity to the inequalities labeling the edges of $P$. The residue inequality of $P$ is thus of the form $a v_1 + b v_{m+1} \succeq c$, where $\succeq$ denotes a strict inequality if and only if at least one of the inequalities labeling the edges of $P$ is strict. Consider for example a path $P$ given by 
\begin{align*}
\Bigl( (x_1,x_2,x_3,x_4), &\left( \{ x_1,x_2 \} ,  x_1 -2x_2  \ge -1   \right),\\
&\left( \{ x_2,x_3 \} ,  x_2 + 3x_3  > -2  \right), \\
&\left( \{ x_3,x_4 \} ,  -x_3  -x_4  \ge 0  \right) \Bigr ),
\end{align*}
we have $x_1 > -1 + 2 ( -2 - 3 x_3) = -5 - 6 x_3 \ge -5 + 6 x_4$ and thus the residue inequality of $P$ is $ x_1 - 6 x_4 > -5$. In the case where $P$ is a loop with initial vertex $v$, its residue inequality is of the form $(a + b) v \succeq c$. If it happens that $(a + b) v > c$, $a + b=0$ and $c \ge  0$ or $(a + b) v \ge c$, $a + b=0$ and $c >  0$, the residue inequality of $P$ is false and we say that $P$ is an \textit{infeasible} loop. Note in particular that infeasibility implies admissibility. We define a \textit{closure} $\overline{\Gamma}_{\Sigma}:= (V_{\Sigma},\overline{\mathrm{E}}_{\Sigma})$ of $\Gamma_{\Sigma}$ to be a graph $\overline{\Gamma}_{\Sigma}$ containing $\Gamma_{\Sigma}$ and having same vertex set, such that $\overline{\mathrm{E}}_{\Sigma}$ is obtained from $\mathrm{E}_{\Sigma}$  by adding for each simple admissible loop $P$ (modulo permutation and reversal) of $\Gamma_{\Sigma}$, a \textit{residue edge} which is a new edge labeled with the residue inequality of $P$. Let moreover $\mathrm{Nontrivial}(\overline{\mathrm{E}}_{\Sigma})$ denote all the elements of $\overline{\mathrm{E}}_{\Sigma}$ that are no self-loop at $x_0$.
Note that a closure is not necessarily unique since the initial vertex of each permutable loop can be chosen arbitrarily. We can now state the main theorem of~\cite{Sho}:

\begin{Thm}\label{Thm:theorem1}
$\Sigma$ is unsatisfiable if and only if $\overline{\Gamma}_{\Sigma}$ has an infeasible simple loop.
\end{Thm}
As an example, consider the system 
\begin{align*}
\Sigma = \{ \Sigma_i \}_{i \in \{1,\dots,6\}} = \Bigl \{ &x_1 - x_2 \ge 0 ,  \  2x_1 + x_2 \ge -1 , \  x_3 - x_1 \ge 0 , \\
 &x_4 - x_3 \ge 0  ,\  x_3 - x_4 \ge - 1 , \ -x_3 \ge 1/2  \Bigr \}.
\end{align*}
It is easy to see that the only loop of $\Gamma_{\Sigma}$ contributing an edge to $\overline{\Gamma}_{\Sigma}$ is the loop 
\[
\Biggl( (x_1,x_2,x_1),(\{x_1,x_2\},\Sigma_1),(\{x_2,x_1\},\Sigma_2)  \Biggr)
\]
 having residue inequality $x_1 \ge -1/3 $. Now note that the loop
\[
\Biggl( (x_0,x_1,x_3,x_0),(\{x_0,x_1\},x_1 \ge -1/3 ),(\{x_1,x_3\},\Sigma_3),(\{x_3,x_0\},\Sigma_6)  \Biggr) \subset \overline{\Gamma}_{\Sigma}
\]
is infeasible and hence $\Sigma$ must be unsatisfiable according to the theorem.


\section{Injectivity of Linear Systems of Inequalities with at most Two Variables per Inequality}

For $j \in I_n=\{1,\dots,n\}$, let $F_j := [-1,1]^{j-1} \times \{ 1 \} \times [-1,1]^{n-j}$ which is a facet of the unit cube $[-1,1]^{n}$. Note that $\mathrm{relint}( F_j) = (-1,1)^{j-1} \times \{ 1 \} \times (-1,1)^{n-j}$.

\begin{Prop}
Let $C \subset \mathcal{Z}_n$ be a convex polyhedral cone with $\mathrm{int}(C) \neq \emptyset$ satisfying
\[
C = \bigcap_{\nu \in \mathcal{N}(C)} \left \{ x \in \R^n : x \cdot (a_{\nu} f_{\nu} + b_{\nu} g_{\nu}) \ge 0 \right \}
\]
with $f_{\nu},g_{\nu} \in \{0\} \cup \{e_1,\dots,e_n\}$ as well as $a_{\nu},b_{\nu} \in \R$ and $f_{\nu} \neq g_{\nu}$. There is then $(j,\tau) \in  I_n \times \{ \pm 1\}$ such that
\[
C \cap \mathrm{relint}(\tau F_{j}) \neq \emptyset = C \cap \mathrm{relint}(-\tau F_{j}).
\]
\end{Prop}
\begin{proof}
We proceed by induction on $n$. It is easy to see that the result holds for $n =1$ and $n=2$. We assume that the result holds for $\{ 1, \dots, n-1\}$ and show that it consequently holds for $n$. Since $\mathrm{int}(C) \neq \emptyset$, there is $(s,\sigma) \in I_n \times \{\pm 1\}$ such that $C \cap \mathrm{relint}(\sigma F_s) \neq \emptyset$. If $C \cap \mathrm{relint}(-\sigma F_s) = \emptyset$, we are done. Hence, assume that 
\begin{equation}\label{eq:equation0}
C \cap \mathrm{relint}(- F_s) \neq \emptyset \neq C \cap \mathrm{relint}( F_s)
\end{equation}
which recalling the notation $\partial H_{e_s} = \{ x \in \R^n : x_s = 0 \}$ implies 
\begin{equation}\label{eq:equation1}
\mathrm{relint}(C \cap \partial H_{e_s}) \neq \emptyset.
\end{equation}
The map $\widehat{\pi}_s$ given by $(x_1,x_2,\dots,x_s,\dots,x_{n}) \mapsto (x_1,x_2,\dots,\widehat{x}_s,\dots,x_{n})$ is, when restricted to $\partial H_{e_s}$, an isometry with the property that $C^0 := \widehat{\pi}_s(C \cap \partial H_{e_s}) \in \mathcal{Z}_{n-1}$. To see that the latter holds, assume without loss of generality that $f_{\nu} \neq e_s$ for every $\nu \in \mathcal{N}(C)$. We can write $\mathcal{N}(C) = \mathcal{N}(C)^{\not{s}} \sqcup \mathcal{N}(C)^{s}$ where $\mathcal{N}(C)^{\not{s}}$ is the set of all $\nu$ such that $f_{\nu} \neq e_s \neq g_{\nu}$ and $\mathcal{N}(C)^{s}$ the set of those such that $f_{\nu} \neq e_s = g_{\nu}$. We then write $C^{\not{s}} :=  \cap_{\nu \in \mathcal{N}(C)^{\not{s}}} H_{\nu}$ and $C^{s} :=  \cap_{\nu \in \mathcal{N}(C)^s} H_{\nu}$ which implies $C = C^s \cap C^{\not{s}}$. It is easy to see that
\[
C \cap \partial H_{e_s} = C^{\not{s}} \cap \partial H_{e_s} \cap \bigcap_{\nu \in \mathcal{N}(C)^s} H_{a_{\nu}f_{\nu}}.
\]
Applying $\widehat{\pi}_s$ on both sides, we get:
\begin{align*}
C^0 
&= \widehat{\pi}_s(C^{\not{s}} \cap \partial H_{e_s}) \cap \widehat{\pi}_s \left(\partial H_{e_s} \cap \bigcap_{\nu \in \mathcal{N}(C)^s} H_{a_{\nu}f_{\nu}} \right) \\
&= \bigcap_{\nu \in \mathcal{N}(C)^{\not{s}}} H_{\widehat{\pi}_s(\nu)} \cap \bigcap_{\nu \in \mathcal{N}(C)^s} H_{\widehat{\pi}_s(a_{\nu}f_{\nu})}
\in \mathcal{Z}_{n-1}.
\end{align*}
It follows by the induction hypothesis that there is $(t,\tau) \in (I_n \setminus \{s\}) \times \{ \pm 1\}$ such that 
\begin{equation}\label{eq:equation3}
C^0 \cap \widehat{\pi}_s(\mathrm{relint}(\tau F_t) \cap \partial H_{e_s} ) \neq \emptyset = C^0 \cap \widehat{\pi}_s(\mathrm{relint}(-\tau F_t) \cap \partial H_{e_s}).
\end{equation}
Note moreover that $C^0 \cap \widehat{\pi}_s(\mathrm{relint}(\tau F_t) \cap \partial H_{e_s} ) \neq \emptyset$ implies $C \cap \mathrm{relint}(\tau F_t) \neq \emptyset$. Furthermore, if $C^{\not{s}} \cap \mathrm{relint}(-\tau F_t) \cap \partial H_{e_s} = \emptyset$, then $C \cap \mathrm{relint}(-\tau F_t) = \emptyset$ and thus we are done. We thus assume that 
\begin{equation}\label{eq:equation31}
C^{\not{s}} \cap \mathrm{relint}(-\tau F_t) \cap \partial H_{e_s} \neq \emptyset.
\end{equation}
We now show that one can find $a,b \in \R$ with $b \neq 0$ such that $C \subset H_{a e_s + b e_t }$. We can assume without loss of generality that in addition to $f_{\nu} \neq e_s$, one has $f_{\nu} \neq e_t$ for any $\nu \in \mathcal{N}(C)$ since otherwise we can find the desired normal vector $a e_s + b e_t $. Let $\mathcal{N}(C)^{\not{t}}$ be the set of all $\nu$ such that $f_{\nu} \neq e_t \neq g_{\nu}$ and $C^{\not{t}} :=  \cap_{\nu \in \mathcal{N}(C)^{\not{t}}} H_{\nu}$. Let
\begin{align*}
W_t := (-1,1)^n \cup \mathrm{relint}( F_t) \cup \mathrm{relint}(-F_t).
\end{align*}
Note that $W_t \cap C \in \mathcal{Z}_n$ and $\widehat{\pi}_t \circ \widehat{\pi}_s(C \cap \partial H_{e_s} \cap \mathrm{relint}(-\tau F_t)) \in \mathcal{Z}_{n-2}$. Let $\Sigma$ and $\Sigma^0$ denote the respective associated systems induced by the supporting half-spaces. Remark that $\Sigma^0$ is obtained by plugging $x_s = 0$ and $x_t = - \tau$ in every inequality of $\Sigma$ and deleting those loops corresponding to those inequalities associated to $W_t$ that are made trivial. Note that $\Sigma^0$ is unsatisfiable by \eqref{eq:equation3} and thus by Theorem \ref{Thm:theorem1}, there is an infeasible (hence by definition admissible) simple loop $L$ in every closure $\overline{\Gamma}_{\Sigma^0}$ of the graph $\Gamma_{\Sigma^0}$ associated to the system $\Sigma^0$.

Let now $\Gamma_{\Sigma^s} := (\mathrm{V}_{\Sigma} \setminus \{x_s,x_t\} ,\mathrm{E}_{\Sigma^s})$ where $\mathrm{E}_{\Sigma^s}$ consists of all labeled edges $E \in \mathrm{E}_{\Sigma^0}$ so that there is $\left( \{ y_{\mu}, x_s \} , a_{\mu} y_{\mu} + b_{\mu} x_s  \succeq c_{\mu}  \right) \in \mathrm{E}_{\Sigma}$ such that $E = \left( \{ y_{\mu}, x_0 \} ,  a_{\mu} y_{\mu}  \succeq c_{\mu}   \right)$ (possibly with $y_{\mu}= x_0$). Analogously, $\Gamma_{\Sigma^t} := (\mathrm{V}_{\Sigma} \setminus \{x_s,x_t\} ,\mathrm{E}_{\Sigma^t})$ where $\mathrm{E}_{\Sigma^t}$ consists of all labeled edges $E \in \mathrm{E}_{\Sigma^0}$ so that there is $\left( \{ y_{\mu}, x_t \} , a_{\mu} y_{\mu} + b_{\mu} x_t  \succeq c_{\mu}  \right) \in \mathrm{E}_{\Sigma}$ such that $E = \left( \{ y_{\mu}, x_0 \} ,  a_{\mu} y_{\mu}  \succeq c_{\mu} + \tau b_{\mu}  \right)$ (possibly with $y_{\mu}= x_0$). Now, it is easy to see that for $u \in \{s,t\}$, one can choose closures satisfying:
\begin{equation}\label{eq:equation4}
\mathrm{Nontrivial} \Bigl(\overline{\mathrm{E}}_{\Sigma^0} \setminus \mathrm{E}_{\Sigma^u} \Bigr) \subset \mathrm{Nontrivial}\Bigl( \overline{ \mathrm{E}_{\Sigma^0} \setminus \mathrm{E}_{\Sigma^u} }\Bigr).
\end{equation}
Indeed, note that since $\mathrm{E}_{\Sigma^u} \subset \mathrm{E}_{\Sigma^0}$, it follows that
\begin{equation}\label{eq:equation5}
\overline{\mathrm{E}}_{\Sigma^0} \setminus \mathrm{E}_{\Sigma^u} = \bigl[ \overline{\mathrm{E}}_{\Sigma^0} \setminus \mathrm{E}_{\Sigma^0} \bigr] \cup \bigl[ \mathrm{E}_{\Sigma^0} \setminus \mathrm{E}_{\Sigma^u} \bigr].
\end{equation}
By definition, $\mathrm{E}_{\Sigma^0} \setminus \mathrm{E}_{\Sigma^u}  \subset \overline{\mathrm{E}_{\Sigma^0} \setminus \mathrm{E}_{\Sigma^u} }$. Now, consider an admissible loop $L_0 \subset \Gamma_{\Sigma^0}$. If $L_0$ contains an edge of $\Gamma_{\Sigma^u}$, then by admissibility (in particular $x_0$ does not arise as intermediate vertex), $L_0$ is a loop starting at $x_0$ and thus $L_0$ does not induce any nontrivial residue edge. Hence, if $L_0 \subset \Gamma_{\Sigma^0}$ induces one nontrivial residue edge in $\overline{\mathrm{E}}_{\Sigma^0} \setminus \mathrm{E}_{\Sigma^0}$, then in particular, the residue equation of $L_0$ does not contain $x_0$ alone and thus $L_0$ does not contain any edge of $\Gamma_{\Sigma^u}$. This thus implies that $\mathrm{Nontrivial}(\overline{\mathrm{E}}_{\Sigma^0} \setminus \mathrm{E}_{\Sigma^0}) \subset \mathrm{Nontrivial}(\overline{\mathrm{E}_{\Sigma^0} \setminus \mathrm{E}_{\Sigma^u}} )$. It finally follows by \eqref{eq:equation5} that \eqref{eq:equation4} holds. 

Now, if $L \subset \overline{\Gamma}_{\Sigma^0}$ is nontrivial and does not contain any edge of $\Gamma_{\Sigma^s}$, we obtain in view of \eqref{eq:equation4}:
\begin{align*}
L &\subset \Bigl(\mathrm{V}_{\Sigma} \setminus \{x_s,x_t\} ,\mathrm{Nontrivial}(\overline{\mathrm{E}}_{\Sigma^0} \setminus \mathrm{E}_{\Sigma^s})\Bigr) \\
&\subset \Bigl(\mathrm{V}_{\Sigma} \setminus \{x_s,x_t\} ,\mathrm{Nontrivial}(\overline{\mathrm{E}_{\Sigma^0} \setminus \mathrm{E}_{\Sigma^s} })\Bigr)
=:\Gamma.
\end{align*}
But $\Gamma$ has $Q = \widehat{\pi}_t \circ \widehat{\pi}_s \left( C^{\not{s}} \cap \mathrm{relint}(-\tau F_t) \cap \partial H_{e_s} \right)$ as associated solution set. Thus $\Gamma$ contains an infeasible simple loop and therefore its associated system is unsatisfiable by Theorem \ref{Thm:theorem1}. Hence, $Q = \emptyset $ and thus $C^{\not{s}} \cap \mathrm{relint}(-\tau F_t) \cap \partial H_{e_s}= \emptyset$, which contradicts \eqref{eq:equation31}. It follows that $L$ has to contain an edge of $\Gamma_{\Sigma^s}$.

Similarly, if $L$ is nontrivial and does not contain any edge of $\Gamma_{\Sigma^t}$, we obtain in view of \eqref{eq:equation4}:
\begin{align*}
L &\subset \Bigl(\mathrm{V}_{\Sigma} \setminus \{x_s,x_t\} ,\mathrm{Nontrivial}(\overline{\mathrm{E}}_{\Sigma^0} \setminus \mathrm{E}_{\Sigma^t})\Bigr) \\
&\subset \Bigl( \mathrm{V}_{\Sigma} \setminus \{x_s,x_t\} ,\mathrm{Nontrivial}(\overline{\mathrm{E}_{\Sigma^0} \setminus \mathrm{E}_{\Sigma^t}})\Bigr)
=:\Gamma.
\end{align*}
But $\Gamma$ has $Q = \widehat{\pi}_t \circ \widehat{\pi}_s \left( C^{\not{t}} \cap W_t \cap \partial H_{e_s} \right)$ as associated solution set. Thus $\Gamma$ contains an infeasible simple loop and therefore its associated system is unsatisfiable by Theorem \ref{Thm:theorem1}. Hence, $Q = \emptyset $ and thus $C^{\not{t}} \cap W_t \cap \partial H_{e_s} = \emptyset$, which contradicts \eqref{eq:equation1} as one can easily see by noting that  $C^{\not{t}} \cap \partial H_{e_s}$ is a cone. It follows that $L$ has to contain an edge of $\Gamma_{\Sigma^t}$.

Finally, note that a self-loop in $\overline{\Gamma}_{\Sigma^0}$ at $x_0$ cannot arise as intermediate segment on an admissible path and no self-loop at $x_i \neq x_0$ can be induced by a loop containing an edge in $\mathrm{E}_{\Sigma^s} \cup \mathrm{E}_{\Sigma^t}$. Hence the only remaining case is when $L$ is a self-loop at $x_0$. But then, since $\Gamma_{\Sigma^0}$ is defined so as not to contain any infeasible self-loop at $x_0$, it follows that $L$ must be induced by a simple nontrivial admissible loop $L_0$ in $\Gamma_{\Sigma^0}$ and as above, one has that $L_0$ needs to be containing an edge of $\mathrm{E}_{\Sigma^s}$ as well as an edge of $\mathrm{E}_{\Sigma^t}$. 

Thus, up to replacing the loop $L \subset \overline{\Gamma}_{\Sigma^0}$ by $L_0$ if necessary, we can assume that $L$ contains an edge in $\mathrm{E}_{\Sigma^s}$ as well as an edge in $\mathrm{E}_{\Sigma^t}$. It follows that $L$ has starting or ending edge, let us say without loss of generality starting edge $E_r = \left( \{ x_0, x_r \} , b_r x_r  \succeq c_r  \right) \in \mathrm{E}_{\Sigma^s}$ and accordingly final edge $E_u = \left( \{ x_0, x_u \} ,  b_u x_u  \succeq c_u   \right) \in \mathrm{E}_{\Sigma^t}$ for some $x_r,x_u \in \mathrm{V} \setminus \{x_0,x_s,x_t \}$. For any edge $E$ of $L$ different from $E_u$ and $E_r$, $E$ does not contain $x_0$ as endpoint by admissibility of $L$, hence
\[
E \in \mathrm{E}_{\Sigma^0} \setminus  (\mathrm{E}_{\Sigma^s} \cup \mathrm{E}_{\Sigma^t} )
\]
which means that $E$ has a corresponding edge $E^{\Sigma} \in \mathrm{E}_{\Sigma}$ that is labeled by the same equation as $E$ and that has thus the same endpoints (these are thus different from $x_s$ and $x_t$). Moreover, by definition of $\mathrm{E}_{\Sigma^u}$, we have edges $E^{\Sigma}_r,E^{\Sigma}_u \in \mathrm{E}_{\Sigma}$ corresponding to $E_r$ and $E_u$ which satisfy $E^{\Sigma}_r = \left( \{ x_r, x_s \} ,  a_r x_r + b_r x_s  \succeq 0   \right)$ and $E^{\Sigma}_u = \left( \{ x_u, x_t \} , a_u x_u + b_u x_t  \succeq 0  \right)$. We thus obtain an admissible (simple) path $P \subset \Gamma_{\Sigma}$ from $x_s$ to $x_t$. The residue inequality of $P$ is then of the form $a x_s + b x_t \succeq 0$ and thus $C \subset H_{\nu}$ for $\nu := ae_s + b e_t$. By \eqref{eq:equation0}, it follows that $|a| < |b|$ and thus by an easy argument $C \cap \mathrm{relint}(- \mathrm{sign}(b) F_t) = \emptyset$. Since we assumed that $C \cap \mathrm{relint}(\tau F_t) \neq \emptyset$, it follows that $\mathrm{sign}(b) = \tau$ and thus
\[
C \cap \mathrm{relint}(\tau F_t) \neq \emptyset = C \cap \mathrm{relint}(- \tau F_t).
\]
This proves the induction step and finishes the proof.
\end{proof}

{\bf Acknowledgements.}
I am very grateful to my PhD advisor Prof. Dr. Urs~Lang for proposing this research topic to me and for reading this work as well as earlier versions. This research was partially supported by the Swiss National Science Foundation.



\end{document}